\documentclass{article}
\usepackage[comma,square,sort&compress, numbers]{natbib}
\usepackage[greek, english]{babel}
\usepackage{amsbsy,amssymb,amscd,amsmath,amsthm,amsfonts}
\usepackage{eurosym}
\usepackage{epsf}
\usepackage{epsfig}
\usepackage{color}
\usepackage{bbold}
\usepackage[T1]{fontenc}
\usepackage[utf8]{inputenc}
\usepackage{authblk}
\usepackage{enumitem}

\newfont{\gothic}{eufm10}

                   \def\R{{\RR}}
\def\RR{{\mathbb{R}}}        \def\N{{\mathbb{N}}}

        \newtheorem{theorem}{Theorem}[section]
\newtheorem{lemma}[theorem]{Lemma}
\newtheorem{proposition}[theorem]{Proposition}
\newtheorem{corollary}[theorem]{Corollary}
\newtheorem{definition1}[theorem]{Definition}

\newenvironment{definition}{\begin{definition1}\rm}{\end{definition1}}

\newtheorem{remark1}[theorem]{Remark}
\newenvironment{remark}{\begin{remark1}\rm}{\end{remark1}}

\newtheorem{example1}[theorem]{Example}

\def\barray{\begin{eqnarray*}}             \def\earray{\end{eqnarray*}}
\def\beq{\begin{equation}} \def\eeq{\end{equation}}
\DeclareMathOperator{\id}{id}

\DeclareMathOperator{\out}{out}
\DeclareMathOperator{\inn}{in}
\DeclareMathOperator{\f}{f}

\makeatletter 
\title{Minimisers of the Allen-Cahn equation and the asymptotic Plateau problem on hyperbolic groups}  
\author{Bla\v z Mramor } 
\affil{\small Freiburg Institute for Advance Sciences (FRIAS), University of Freiburg, Germany. }
\date{}

\begin{document}  \hyphenation{asympto-ti-cally}

\newcommand{\p}{\partial}
\maketitle

\noindent

\abstract{\noindent 
We investigate the existence of non-constant uniformly-bounded minimal solutions of the Allen-Cahn equation on a Gromov-hyperbolic group. We show that whenever the Laplace term in the Allen-Cahn equation is small enough, there exist minimal solutions satisfying a large class of prescribed asymptotic behaviours. For a phase field model on a hyperbolic group, such solutions describe phase transitions that asymptotically converge towards prescribed phases, given by asymptotic directions.
In the spirit of de Giorgi's conjecture, we then fix an asymptotic behaviour and let the Laplace term go to zero. In the limit we obtain a solution to a corresponding asymptotic Plateau problem by $\Gamma$-convergence.}

\section{Introduction}

Consider a group $G$ with a fixed finite symmetric set of generators $S$. The Cayley graph $C(G,S)$ of $G$ with respect to $S$ with the word metric is a complete metric space. We assume that $G$ is word hyperbolic, i.e.~there exists a $\delta>0$, such that $C(G,S)$ is a $\delta$-hyperbolic metric space in the sense of Gromov. There is a natural way of defining the {\it boundary at infinity of $G$}, denoted by $\p G$, such that the space $G\cup \p G$ metrized by the so-called {\it visual metric} becomes a compact metric space. This is explained more precisely in section \ref{geometric_groups}, where we also present some basic properties of these objects, together with some references to the relevant literature. We study the Allen-Cahn equation on $G$.

The Allen-Cahn equation is defined as follows. Let $V:\R\to \R$ be a Morse function with two absolute minima $c_0,c_1\in \R$, with value bounded away from all other critical values, and let us denote for every function $x:G\to\R$ the discrete Laplace operator by $$(\Delta x)_g:=\sum_{s\in S}(x_{gs}-x_g) \ .$$ The Allen-Cahn equation for any $\rho>0$ is then given by \begin{equation}\label{rr} \rho(\Delta x)_g-V'(x_g)=0 , \text{ for all } \ g\in G .\end{equation} 

The equation (\ref{rr}) comes with a variational structure given by the formal action functional \begin{equation}\label{fe}W(x):=\sum_{g\in G}\sum_{s\in S}\left(\frac{\rho}{4}(x_{gs}-x_g)^2+V(x_g)\right),\end{equation} which has a well defined gradient $\nabla_gW(x)=\rho(\Delta x)_g-V'(x_g)$ (see also section \ref{var_prob}). In the literature, one often takes the potential to be $\displaystyle V(y):=\frac{1}{4}(1-y^2)^2$. As is usual in the calculus of variation, we call an entire solution $x:G\to \R$ of (\ref{rr}) a global minimiser, if it minimises the action (\ref{fe}) globally, i.e., if for every finitely supported variation $v:G\to \R$ of $x$ $$W(x+v) \geq W(x) $$ (this is explained more precisely in definition \ref{def_gm}). 

The first goal of this paper is to construct a large class of uniformly bounded global minimisers of the Allen-Cahn equation (\ref{rr}), which have prescribed asymptotic behaviour, tending to either $c_0$ or $c_1$ along infinite geodesic rays in $G$. 

For technical reasons, we only consider the equation (\ref{rr}) when the constant $\rho$ is sufficiently small. In this case, as explained in section \ref{anticontinuum}, solutions of (\ref{rr}) can be obtained as continuations from the so-called anti-continuum limit by an implicit-function-type argument. They are in the supremum norm close to a function $x^0:G\to \R$, such that for every $g\in G$, $x^0_g=c$ for some critical point $c$ of $V$ which may depend on $g$. 
It follows from the results in section \ref{var_prob} that if $\rho$ is sufficiently small and $x$ is a global minimiser such that $x_g$ is contained in the interval $[c_0,c_1]$, then for every $g\in G$ the value $x_g$ is uniformly close to either $c_0$ or $c_1$. This observation simplifies our analysis of global minimisers. 

Our main result, given in theorem \ref{dirichlet}, is to construct global minimisers of the Allen-Cahn equation (\ref{rr}) with such small constant $\rho$, which have prescribed asymptotics. More precisely, we solve the {\it minimal Dirichlet problem at infinity}, defined as follows.

\begin{definition}\label{dp_def}
Given a set $D_0\subset \p G$ with $\overline{\mathring{D}}_0=D_0$, let $D_1:=\overline{\p G\backslash D_0}$. We say that a global minimiser $x$ solves the minimal Dirichlet problem at infinity for (\ref{rr}) on $D_0$ and $D_1$, if the following holds. For every $\xi \in \mathring{D}_j$ where $j\in \{0,1\}$ and for every $\epsilon>0$, there exists a neighbourhood $\mathcal{O}_\xi\subset G\cup \p G$ of $\xi$ in the topology given by the visual metric, such that $|x_g-c_j|<\varepsilon$ for all $g\in \mathcal{O}_\xi\cap G$.


\end{definition}

The solutions that we obtain by solving the minimal Dirichlet problem at infinity, can be used to solve a version of the asymptotic Plateau problem. This is the content of section \ref{plateau_section}. We show that given two sets $D_0, D_1\subset \p G$, which are as in definition \ref{dp_def}, there exist corresponding sets $\mathcal{D}_0\subset G$ and $\mathcal{D}_1\subset G$ with $\mathcal{D}_0\cap\mathcal{D}_1=\varnothing$ and $\mathcal{D}_0\cup \mathcal{D}_1=G$, which asymptotically converge to $\mathring{D}_0$ and $\mathring{D}_1$, respectively, and such that they satisfy a particular minimality property. More precisely, the number of edges in the Cayley graph $C(G,S)$, which connect $\mathcal{D}_0$ to $\mathcal{D}_1$ is minimal with respect to finite perturbations. See definition \ref{pp_def} and theorem \ref{plateau} for the exact statement.

\subsection{Motivation}

The Allen-Cahn equation first appeared in the study of the phase field models on the Euclidean space and gained its popularity in geometry due to the work on de Giorgi's conjecture, which connects it to the study of minimal hyperplanes. In both cases one is interested in minimal solutions of (\ref{rr}) with small constants $\rho$. 

In general, two obvious global minimisers of the Allen-Cahn equation exist, namely, $x\equiv c_0$ and $x\equiv c_1$. In the study of entire solutions, the interesting research questions revolve around describing the set of global minimisers that are (due to the ellipticity of the differential operator) trapped between the graphs of $c_0$ and $c_1$ and asymptotically connect them. When one considers the Allen-Cahn equation on hyperbolic spaces, one expects to have many more such global minimisers than in the Euclidean setting. 

This is specifically motivated by the Laplace equation, the prototypical case of a variational elliptic problem. In this case it is well known that all harmonic functions are global minimisers. In the hyperbolic setting, one approach to obtain non-constant entire solutions is the extensively studied Dirichlet problem at infinity for harmonic functions. Contrary to the Euclidean case, where all entire harmonic functions are constant, the solutions to the Dirichlet problem at infinity on hyperbolic metric spaces give us a wealth of harmonic functions with asymptotic behaviour prescribed by a function on the boundary at infinity. On manifolds with pinched negative sectional curvature such solutions have first been constructed by Choi in \cite{choi} and by Anderson in \cite{anderson83}, by the method of barriers. They solved the Dirichlet problem at infinity given by any continuous function on the boundary at infinity. Further solutions were obtained by Sullivan in \cite{sullivan79, sullivan83} for more general boundary conditions using a probabilistic approach. In the context of hyperbolic groups, this problem was first studied by Ancona in \cite{ancona}. There are numerous generalisations of these results, on one hand focusing on more general operators which have all constant functions as solutions, and on the other hand focusing on more general underlying spaces. For a recent state of the art result see e.g. \cite{holopainen_lang}. An important observation for us is that neither the approach of barriers, which uses the foliation by constants in a crucial way, nor a probabilistic approach using random walks, are applicable to solving the Allen-Cahn equation. 

The focus on finding minimal solutions and not general solutions is on one hand driven by physical motivation, where one searches for energy-minimising states in the phase field models, and on the other hand by the work on de Giorgi's conjecture, relating minimal solutions to minimal hypersurfaces. Another reason to focus on global minimisers and not general solutions is the observation, explained in detail in section \ref{anticontinuum}, that for small enough constants $\rho>0$ in the Allen-Cahn equation (\ref{rr}) on hyperbolic groups, one can easily find a wealth of solutions via a version of the implicit function theorem. Thus the only difficult question is whether non-constant global minimisers exist. 

For hyperbolic manifolds with constant sectional curvature similar results to those presented in this paper have been obtained by Pisante and Ponsiglione in \cite{italians} and by Mazzeo and Saez in \cite{mazzeo}. In their proofs the symmetry of the hyperbolic space plays a crucial role and allows them to reduce the PDE to an ODE. Other existence results about minimisers for more general nonlinear variational elliptic operators on quite general groups have been discussed in \cite{llave-lattices,candel-llave}, but they discuss minimisers on the abelianisation of the group, which are not necessarily global minimisers of the group itself. For hyperbolic groups we thus had to develop a new approach, which is based on the variational structure of the equation and is in the spirit of de Giorgi's conjecture. 

To solve the minimal Dirichlet problem at infinity we construct global minimisers with prescribed asymptotic behaviour as limits of functions $x^n$ for $n\in \N$, where for every $n$, $x^n$ is a minimal solution of a Dirichlet problem on a ball of radius $n$, with boundary values given by a function with values in $c_0$ and $c_1$. The main problem is then to show that the functions $x^n$ do not converge to one of the constant solutions $c_0$ or $c_1$. We conjecture that this happens in the case that the group $G$ is amenable and the interior of $D_0$ is quite small compared to the interior of $D_1$, with respect to the Patterson-Sullivan measure (see section \ref{PS-section}). 
Note that the the choice of a small constant $\rho$ is not necessary for the existence of minimisers $x^n$, which in general follows from the coercivity of the equation (\ref{rr}), as noted in lemma \ref{existence_lemma}. The condition that $\rho$ is small, however, gives us very precise control of the behaviour of minimisers. It might be possible to obtain similar results for general constants $\rho$, but such results are not yet within our reach. 

Using the solutions obtained by solving the minimal Dirichlet problem at infinity, we solve a corresponding asymptotic Plateau problem using a simple version of $\Gamma$-convergence for vanishing $\rho$. In the setting of riemannian manifolds, the asymptotic Plateau problem poses the question of finding area-minimising submanifolds with prescribed asymptotic behaviour. In the Euclidean case, such solutions are fairly restricted, as showed by the work on de Giorgi's conjecture. Indeed, in dimension smaller than nine, all such solutions are affine hyperplanes (see \cite{delpino2} for an overview of related results). For hyperbolic manifolds, on the other hand, the class of solutions is vastly greater. The first results in this direction were obtained by Anderson for the hyperbolic space $\mathbb{H}^n$ in \cite{anderson83a}. For Gromov-hyperbolic riemannian manifolds the existence of solutions to the asymptotic Plateau problem was proved by Lang in \cite{lang03}, based on a Morse-type lemma for quasi-minimal hypersurfaces given by Bangert and Lang in \cite{bangert_lang}. 
In the special case of codimension one, absolutely minimal submanifolds are closely related to globally minimal solutions of the Allen-Cahn equation, and are obtained by $\Gamma$-convergence with the parameter $\rho\to 0$. In the Euclidean setting, this was conjectured by de Giorgi and proved by Modica in \cite{modica} (see also \cite{delpino}). The same approach was taken for the hyperbolic space in \cite{italians}. We extend these results to hyperbolic groups, by replacing the concept of minimal hypersurfaces of codimension one by an appropriate notion of minimality of the boundaries of sets.

\subsection{Outline}
Our paper and proofs are structured as follows. Section \ref{geometric_groups} is focused on hyperbolic groups. We review some basics about their boundary and growth, and define geometric objects called cones, which generate a ``cone topology'' equivalent to the visual metric topology, and which is similar to the cone topology used in the study of the Dirichlet problem at infinity in \cite{choi}. In section \ref{var_prob} we discuss the variational structure of the Allen-Cahn equation, the notion of minimality, and the existence of solutions via the so-called anti-continuum limit the proof of which is contained in the appendix. In section \ref{DP}, this allows us to split the group $G$ in two subsets, one where the global minimiser is approximately $c_0$, and the other where it is approximately $c_1$. The set of points of $G$, which are at most distance one to both of these sets is called the transition set and is in a particular sense quasi-minimal. We use growth estimates and an isoperimetric profile estimate for the hyperbolic group to show that this transition set extends uniformly towards the identity. The bounds on the distance from the set to the identity depend in principle only on the hyperbolicity constant $\delta$ and the largest radius of a ball at infinity which is fully contained in $D_0$ or $D_1$. This allows us to prove the existence of minimisers with prescribed asymptotics. Moreover, the fact that the bound is independent of the constant $\rho$ lets us take the limit $\rho\to 0$ and use a $\Gamma$-convergence type of argument to solve the asymptotic Plateau problem in section \ref{plateau_section}.

\subsection{Acknowledgement }
I would like to thank Prof.~V.~Bangert for the helpful conversations and for his valuable comments.\\ 

\noindent
This paper was written during the author's stay at the Freiburg Institute for Advance Sciences in the scope of the FRIAS COFUND fellowship programme. The research was funded from People Programme (Marie Curie Actions) of the European Union's FP7/2007-2013, REA grant agreement nb. 609305.


\section{Hyperbolic groups}\label{geometric_groups}

In this section we gather some basic facts about hyperbolic groups. The general framework of hyperbolic groups was first introduced by Gromov in \cite{gromov1} and further developed in \cite{gromov2}. There exists an extensive literature on the subject. For more thorough overviews, we refer the reader to \cite{harpe,notes1,notes2}. Typical examples are fundamental groups of manifolds with strictly negative sectional curvature and free groups.

Let $\mathcal{K}=C(G,S)$ denote the Cayley graph of a group $G$, which has a finite symmetric set of generators $S$. Setting the length of the edges to one, the word metric $d(\cdot,\cdot)$ is defined as follows: for every two points $g,\tilde g\in G$, $d(g,\tilde g)$ is the least length of a path in $\mathcal{K}$ connecting $g$ and $\tilde g$. A curve $\gamma_{g,\tilde g}$ from $g$ to $\tilde g$ of length $l(\gamma_{g,\tilde g})=d(g,\tilde g)$ is called a geodesic and needs not to be unique (the analogous concept is a minimal geodesic in the riemannian case). With the word metric, the space $(\mathcal{K}, d)$ becomes a complete geodesic metric space, i.e. geodesics between any two points exist. We denote for every $g\in G$ its word length by $|g|:=d(\id,g)$.

We assume that $G$ is word-hyperbolic, that is, we assume that $(\mathcal{K},d)$ is a $\delta$-hyperbolic metric space, for some constant $\delta>0$. One way to describe such spaces is to say that every geodesic triangle (i.e. a collection of three points $g_1,g_2,g_3\in G$ together with three geodesics $\gamma_{g_1,g_2},\gamma_{g_2,g_3},\gamma_{g_3,g_1}\subset \mathcal{K}$) is $\delta$-slim (i.e. every geodesic ``side'' is in the $\delta$-neighbourhood of the other two geodesic sides). Furthermore, we assume that $G$ is non-elementary, i.e.~$G$ is not finite and does not contain an infinite cyclic subgroup of finite index. 

Hyperbolicity is independent of the generating set, but the topology of $\mathcal{K}$ is not. For this paper we thus choose a fixed symmetric generating set $S$, i.e. $S=S^{-1}$. 
The distance $d$ on the Cayley graph $\mathcal{K}$ restricts to a distance on the group $G\subset \mathcal{K}$, which we also call $d(\cdot,\cdot)$.

It is well known that non-elementary word-hyperbolic groups exhibit exponential growth and that they are non-amenable. The precise statements about their behaviour which we need in our proofs, together with some references to the relevant literature, are collected below.

\subsection{Boundaries of sets and the isoperimetric profile}

Let us define for a set $\mathcal{B}\subset G$ its ``outer'' and its ``inner'' set $$\mathcal{B}^{\out}:=\{g \in \mathcal{K} \ | \ gs \in \mathcal{B} \text{ for some } s\in S\} \ \text{ and }\ \mathcal{B}^{\inn}:=\{g \in \mathcal{B} \ | \ gs \in \mathcal{B} \text{ for all } s\in S\}.$$ Furthermore, we define the outer and the inner boundaries of $\mathcal{B}$ by $\p^{\out} \mathcal{B}:=\mathcal{B}^{\out} \backslash \mathcal{B}$ and $\p^{\inn} \mathcal{B}:=\mathcal{B} \backslash \mathcal{B}^{\inn}$ and the ``full'' boundary by $\p^{\f} \mathcal{B}:=\p^{\out} \mathcal{B}\cup\p^{\inn} \mathcal{B}$. Since $d(g,h)=1$ if and only if $h=gs$ for some $s\in S$ (then obviously also $g=hs^{-1}$), an equivalent definition for the set $\mathcal{B}^{\out}$ is $\mathcal{B}^{\out}=\{g \in \mathcal{K} \ | d(g,\mathcal{B})\leq 1\}$, from which it easily follows that \begin{equation}\label{boundary}\p^{\out}(\mathcal{B}\cap \mathcal{D})=(\p^{\out} \mathcal{B}\cap  \mathcal{D}^{\out})\cup (\mathcal{B}^{\out}\cap \p^{\out} \mathcal{D}) \end{equation} for any two set $\mathcal{B},\mathcal{D}\subset G$.

The following almost linear isoperimetric inequality holds. It was first proved by Coulhon in \cite{coulhon}, where the author investigates the behaviour of random walks on hyperbolic groups. 

\begin{lemma}\label{isoperimetric}
There exists a constant $k_0$ such that for every finite set $\mathcal{B}\subset G$ we have $$\frac{\#\mathcal{B}}{\log(\#\mathcal{B})}\leq k_0\cdot\#(\p^{\out}\mathcal{B}) \ ,$$ where $\# \mathcal{B}$ denotes the number of points contained in the set $\mathcal{B}$
\end{lemma}


\subsection{The boundary at infinity}

Nice overviews of results on the boundary at infinity can be found in \cite{kapovich} and \cite{calegari}. 

\subsubsection{Definition of the boundary at infinity}

A geodesic ray is given by an isometry $\gamma:[0,\infty)\to \mathcal{K}$, where $[0,\infty)$ is equipped with the standard metric. On $\delta$-hyperbolic spaces, one may define an equivalence relation on the set of geodesic rays, by $\gamma_1\sim \gamma_2$, if there exists a $C\in \R$ such that $d(\gamma_1(t),\gamma_2(t))\leq C$ for all $t\in \R_+$. One then defines the boundary at infinity $\p \mathcal{K}$ as the set of such equivalence classes of rays.
Defining $\overline{\mathcal{K}}:=\mathcal{K}\cup\p\mathcal{K}$, we may for every point $\xi\in \p\mathcal{K}$, extend any ray $\gamma\in \xi$ to $\gamma:[0,\infty]\to \overline{\mathcal{K}}$, by defining $\gamma(\infty)=\xi$. It then holds that for every $g\in \mathcal{K},\xi \in \p \mathcal{K}$ there exists a geodesic ray $\gamma_{g,\xi}\subset \overline{\mathcal{K}}$, such that $\gamma_{g,\xi}(0)=g$ and $\gamma_{g,\xi}(\infty)=\xi$. Moreover, for any two points $\xi, \mu \in \p \mathcal{K}$ there exists an (infinite) geodesic $\gamma_{\xi,\mu}\subset \overline{\mathcal{K}}$, connecting these two points at infinity, i.e. the ray $\gamma_{\xi,\mu}|_{[0,-\infty)}$ belongs to the equivalence class $\xi$ and $\gamma_{\xi,\mu}|_{[0,\infty)}$ to $\mu$. With these definitions, geodesic triangles in $\overline{\mathcal{K}}$ are also $\tilde \delta$ slim for a uniform constant $\tilde \delta$ (see \cite{kapovich}). For $G,S$ as above, we define the boundary $\p G$ of the group $G$ by $\p G:=\p\mathcal{K}$.

\subsubsection{The visual metric}

The so-called visual metric makes $\overline{\mathcal{K}}$ into a compact geodesic metric space. It is defined for every $y,\tilde y\in \mathcal{K}$ by $$d_\varepsilon(y,\tilde y):= \inf_{p_{y,\tilde y}}\int_0^{l(p_{y,\tilde y})}e^{-\varepsilon \cdot d(\id, p_{y,\tilde y}(s))}ds \ ,$$ where $p_{y,\tilde y}\subset \mathcal{K}$ is a path from $y$ to $\tilde y$. Let $\xi,\mu\in \p \mathcal{K}$ and let $y\in \gamma_1$ where $[\gamma_1]= \xi$ and $\tilde y\in \gamma_2$ where $[\gamma_2]= \mu.$ It follows from $\delta$-hyperbolicity that there exists an $\varepsilon_0>0$, such that for all $0<\varepsilon\leq\varepsilon_0$, the limit of $d_\varepsilon(y,\tilde y)$ for $|y|,|\tilde y|\to \infty $ is uniformly bounded and that it depends only on the equivalence classes $\xi$ and $\mu$. This gives us for every $0<\varepsilon\leq \varepsilon_0$ a metric on $\overline{\mathcal{K}}$, for which there exists a constant $\lambda>0$, such that for all $\xi,\mu\in \p \mathcal{K}$ and every infinite geodesic $\gamma_{\xi,\mu}$\begin{equation}\label{vm_estimate}\lambda^{-1}e^{-\varepsilon d(\id, \gamma_{\xi, \mu})}\leq  d_\varepsilon(\xi, \mu)\leq \lambda e^{-\varepsilon d(\id, \gamma_{\xi, \mu})}. \end{equation} 

Obviously, the restriction of $d_\varepsilon$ to $G\cup \p G$, gives us also a compact metric space. We shall use the following notations for balls in $G$, $\p G$, or in $G\cup \p G$. For every $n\in N$ and $r\geq 0$, and for every $g_0\in G$, $\xi_0\in \p G$ and $y_0 \in G \cup \p G$ define
\begin{equation}\begin{aligned}
&\mathcal{B}_n:=\{g\in G \ | \ |g|\leq n\}, \ \mathcal{B}_n(g_0):=\{ g\in G \ | \ d(g_0, g)\leq n\} \text{ and }  \mathcal{S}_n:=\{g\in G \ | \ |g|= n\},\\
&B^\varepsilon_r(\xi_0):=\{\xi \in \p G \ | \ d_\varepsilon(\xi_0,\xi)\leq r \}, \\
&\mathcal{B}^\varepsilon_r(y_0):=\{y \in G \cup \p G \ | \ d_\varepsilon(y_0,y)\leq r \}.
\end{aligned}\end{equation}


\subsection{Growth of the group and the topology of the boundary}

The following section describes some results on the growth and the entropy of $(G,S)$. Furthermore, it introduces the so called Patterson-Sullivan measure, a doubling measure with respect to the visual metric, which is supported on the boundary at infinity, and a concept of dimension of the boundary at infinity. These results are mostly due to Coornaert, based on results of Cannon, see e.g. \cite{echlpt} or \cite{calegari} for nice overviews.

\subsubsection{Entropy and Cannon's theory}

Let $(G,S)$ be as above. 
The {\it entropy} of the group $G$ with the generating set $S$ is then defined by $$h(G,S):=\limsup_{n\to \infty} \frac{1}{n}\log(\#(\mathcal{B}_n)).$$ 

An estimate for the entropy can be obtained by analysing the formal growth function for $(G,S)$, which is defined by $s(t):=\sum_n (\#\mathcal{S}_n)t^n$. Cannon's methods from \cite{cannon}, where the word problem for discrete co-compact hyperbolic space groups is solved, show that general hyperbolic groups may be viewed as finite state automata, which implies that $s(t)$ is rational. As a consequence,  a non-elementary word-hyperbolic group $G$ with the generating set $S$ grows uniformly exponentially, i.e. there exists a $C>0$, an integer $k\geq 0$, and an $h>0$, such that \begin{equation}\label{cannon}C^{-1}e^{h\cdot n}n^k\leq \#(\mathcal{B}_n)\leq C e^{h\cdot n}n^k.\end{equation} In particular we obtain that $h(G,S)=h>0$.

\subsubsection{The Patterson-Sullivan measure}\label{PS-section}
Using the entropy of the group and the methods that Sullivan and Patterson developed in \cite{sullivan79} and \cite{patterson}, Coornaert investigated in \cite{coornaert93} the topological properties of $\p G$. He defined a probability measure $\nu$, supported on $\p G$, as follows. 
Define for every $s>h=h(G,S)$ a probability measure $\nu_s$ supported on $G$, as a sum of Dirac measures at every $g\in G$ with weights $$\frac{e^{-s|g|}}{\sum_{g\in G}e^{-s|g|}} \ . $$ Using the uniform estimate (\ref{cannon}) for the growth rate of the size of $\mathcal{B}_n$ and by compactness of the space of probability measures on $G\cup \p G$, he argued that there exists a limit measure $\nu$ for $s\to h$, which is supported on $\p G$. This is referred to as the {\it Patterson-Sullivan measure}. The boundary at infinity with the visual metric and the Patterson-Sullivan measure is a metric measure space. Coornaert showed furthermore that for all $g\in G$ the pullback $g_*\nu$ is absolutely continuous with respect to $\nu$ and that the Radon-Nikodym derivative $d(g_*\nu)/d\nu$ is uniformly bounded. For this, he developed  in \cite{coornaert-p} a theory about horospheres in hyperbolic groups and their correspondence to the boundary at infinity. Roughly put, the bound behaves like $e^{-h  \cdot d(g,\mathcal{H})}$, where $\mathcal{H}$ is a horosphere corresponding to the appropriate point at infinity. This uniform bound can be used to analyse the correspondence between the measure and the metric.


\subsubsection{Metric properties of the Patterson-Sullivan measure}

Still following Coornaert in \cite{coornaert93}, the behaviour of the measure $\nu$ with respect to the $\varepsilon$-metric on $\p G$ can be further analysed via the concept of a shadow, originally due to Sullivan.

\begin{definition}
For every $R>0$ and $g\in G$ we define the $R$-shadow of $g$ from $\id$ by $$S(g,R):=\{\xi\in \p G \ | \ \forall \gamma\in \xi \text{ with } \gamma(0)=\id, d(\gamma,g)\leq R\}\subset \p G.$$ 
\end{definition}

In the next section, we shall use shadows to define cones, which one may view as generalisations of the cones in the definition of the topology for the Dirichlet problem at infinity in riemannian manifolds (see \cite{choi, anderson83,anderson_schoen}). The following proposition explains that shadows with a large enough $R$ turn out to be ``almost-round'' and can be used to show that $\p G$ with the measure $\nu$ and the visual metric is a doubling metric measure space. In fact, even more is true, $\nu$ is a so-called $D$-measure, with $D=h/\varepsilon$.

Let us define for every $\xi\in \p G$ the set $\xi_{\id}\subset G$ as the union of all the points $g\in \gamma$, where $\gamma\subset \overline{\mathcal{K}}$ is any ray such that $\gamma(0)=\id$ and $\gamma(\infty)=\xi$. That is, \begin{equation}\label{xiid}\xi_{\id}:=\bigcup_{\gamma=\gamma_{id,\xi}} (\gamma\cap G).\ 
\end{equation}

Recall the definition of the hyperbolicity constant $\tilde \delta$ for the space $\overline{\mathcal{K}}$.

\begin{proposition}\label{shadow}
There exists an integer $R\geq 2 \tilde \delta$, which we now fix for the rest of this text, such that the following assertions are true:

\begin{enumerate}
\item[(1)] There exists a constant $C_1>0$ such that for all $g\in G$ the measure of the shadow $S(g,R)$ is bounded by $$C_1^{-1}e^{-h |g|}\leq \nu(S(g,R))\leq C_1e^{-h |g|}.$$
\item[(2)] For every $\xi\in \p G$, $r>0$ and $g_1,g_2\in \xi_{\id} \subset G$ there exists a constant $C_2>0$ such that $$\begin{aligned}B_r^\varepsilon(\xi) \subset S(g_1,R)& \ \text{ if } \ C_2 e^{-\varepsilon|g_1|}\geq r \text{ and }\\ S(g_2,R)\subset B_r^\varepsilon(\xi)& \ \text{ if } \ C_2^{-1} e^{-\varepsilon|g_2|}\leq r.\end{aligned}$$
This implies that there exists a constant $C_3$ such that for every $\xi \in \p G$ and for every $r>0$ the measure of the ball $B^\varepsilon_r(\xi)$ is bounded by $$C_3^{-1}r^D\leq \nu(B^\varepsilon_r(\xi))\leq C_3r^D,$$ where $D:=h/\varepsilon$. (In particular, $\nu$ is a doubling measure.)
\item[(3)] For every $g\in G$ and $\xi\in S(g,R)$ there exists a constant $C_4>0$ such that $$S(g,R)\subset B_r^\varepsilon(\xi)\ \text{ if } \ r\geq C_4 e^{-\varepsilon |g|}.$$ 

\end{enumerate}
\end{proposition}

\begin{proof}We present only proofs of (2) and (3), as they are slightly more precisely phrased than in the references. For a proof of (1) we refer the reader to \cite{calegari}. 

\begin{enumerate}

\item[(2)] Let now $\xi,\eta\in \p G$, $r>0$ and let $d_\varepsilon(\xi,\eta)<r$. Denote by $\gamma_{\xi,\eta}$ a geodesic that connects $\xi$ to $\eta$ in $\mathcal{K}$. Then by (\ref{vm_estimate})  $$\lambda^{-1} e^{-\varepsilon d(\gamma_{\xi,\eta},\id)}\leq d_\varepsilon(\xi,\eta)\leq r.$$ Let $g_1\in \xi_{\id}$ with $ \lambda^{-1} e^{-\varepsilon(|g_1|+ \tilde \delta)}\geq r$. Since geodesic triangles in $\overline{\mathcal{K}}$ are $\tilde \delta$-slim it follows that $d(g_1,\eta_{\id})\leq \tilde \delta$. Because $R_0\geq 2 \tilde \delta$, it follows that $\eta\in S(g_1,R)$ for all $R\geq R_0$ so the first inclusion follows.

The proof of the second inclusion is very similar and we only sketch the idea. If $g_2\in \xi_{\id}$ with $|g_2|$ small, then $d(\gamma_{\xi,\eta},\id)\geq |g_2|+ \tilde \delta$ as soon as $d_\varepsilon(\xi,\eta)<r$.

To prove the estimate for the measure of metric balls, let $\xi, r, g_1, g_2$ be as above and assume that $g_1\in \xi_{\id}$ is such that $|g_1|$ is the largest integer satisfying the inequality $\displaystyle r\leq C_2e^{-\varepsilon (|g_1|)}$. Let $N\in \N$ be such that $e^{\varepsilon N}\geq C_2^{-2}$. Then $$r\geq C_2e^{-\varepsilon (|g_1|+1)}\geq C_2^{-1}e^{-\varepsilon N}e^{-\varepsilon (|g_1|+1)}\geq C_2^{-1} e^{-\varepsilon (|g_1|+1+N)}.$$ Choosing $g_2\in \xi_{\id}$ such that $|g_2|=|g_1|+1+N$, we get $\nu(S(g_2,R))\leq \nu(B^\varepsilon_r(\xi))\leq \nu(S(g_1,R))$ by inclusion. Together with (1) this implies that $$C_1^{-1}C_2^{-D}e^{-h(N+1)}r^D\leq C_1^{-1}
e^{-h |g_2|}\leq \nu(B^\varepsilon_r(\xi))\leq C_1e^{-h |g_1|}\leq C_1e^h C_2^{D}r^D.$$

\item[(3)] Let $\xi,\eta \in S(g,R)$ and choose points $g_\xi \in \mathcal{B}_R(g)\cap \xi_{\id}$ and $g_\eta\in \mathcal{B}_R(g)\cap \eta_{\id}$. By the triangle inequality $$d_\varepsilon(\xi,\eta)\leq d_\varepsilon(\xi,g_\xi)+ d_\varepsilon(g_\xi, g_\eta)+d_\varepsilon(g_\eta,\eta).$$ Since $d(g_\xi, g_\eta)\leq 2R$ and $g_\xi, g_\eta\in \mathcal{B}_R(g)$, it follows that $d(\id, \gamma_{g_\xi, g_\eta})\leq |g|-2R$ and by (\ref{vm_estimate}) that $$d_\varepsilon(\xi,\eta)\leq \lambda e^{-\varepsilon |g_\xi|}+\lambda e^{-\varepsilon (|g|-2R)}+ \lambda e^{-\varepsilon |g_\eta|}\leq 3\lambda e^{2R} e^{-\varepsilon|g|}.$$
\end{enumerate}
\end{proof}

\begin{remark} As a corollary of statement (1) in the proposition it is not difficult to obtain a better estimate than (\ref{cannon}) for the size of $n$-balls in $G$, namely, there exists a uniform constant $\tilde C>0$, such that \begin{equation}\label{cannon_improved}\tilde C^{-1}e^{h\cdot n}\leq \#(\mathcal{B}_n)\leq \tilde C e^{h\cdot n}.\end{equation} Moreover, by taking coverings of sets, it is possible to extend the measure estimate for the balls to general measurable sets. Indeed, there exists a constant $C_0$, such that for all $\nu$-measurable sets $A\in \p G$, $$C_0^{-1}H^D(A) \leq \nu(A) \leq C_0H^D(A),$$ where $H^D$ denotes the $D$-dimensional Hausdorff measure. For proofs see \cite{calegari}.

Note also that by proposition \ref{shadow}, there are no isolated points in $\p G$. Moreover, because the action of $G$ on $\p G$ is quasi-conformal it follows that $\p G$ is either a sphere or a Cantor set (see \cite{calegari} for details).
\end{remark}

\subsection{Cones in Cayley graphs} 

This section is devoted to the definition and analysis of cones. These are specific subsets of the group, useful for constructing a topology equivalent to the visual metric in analogy to the cone topology used on riemannian manifolds of negative sectional curvature, such as in \cite{choi}. We define them using the boundary at infinity and develop estimates about their growth and their boundaries. The main lemma of this section shows that cones grow uniformly exponentially.

\begin{definition}\label{cone}
Let $U\subset \p G$ be a set and let $R$ be as in proposition \ref{shadow}. Define the $U$-cone by $$\mathcal{C}_{U}:=\{ g \in G \ |\ S(g,R) \subset U\}.$$ 
\end{definition}
Note that by proposition \ref{shadow}, $\mathcal{C}_U= \varnothing$ when $\mathring{U}=\varnothing$. 
%
The cone $\mathcal{C}_{U}$ for a set $U\subset \p G$ is a natural choice of a geometric object, because we can analyse its growth properties by the known properties of a shadow obtained in proposition \ref{shadow}. 

For the remainder of this section, we shall focus on $U$-cones, where $U$ is a metric ball at infinity.

\subsubsection{Truncated cones and the visual metric}
With the minimal Dirichlet problem in mind, we are interested in the behaviour of functions $x:G\to \R$ close to $\p G$. Let $x^n:G\to \R$ be a sequence of functions, for which we would like to analyse convergence near $U\subset \p G$. The following lemma shows that uniform convergence on truncated cones $C_{B^\varepsilon_r(\xi_0)}\backslash \mathcal{B}_n$ corresponds to uniform convergence with respect to (half-)balls $\mathcal{B}^\varepsilon_{\tilde r}(\xi_0)\subset \overline{G}$. 

\begin{lemma}\label{topology} For all $r>0$ small enough, there exists a constant $c>0$, such that for every metric ball at infinity $B^\varepsilon_r(\xi_0)\subset \p G$, there is an $n\in \N$ with 
$$(\mathcal{B}^\varepsilon_{ c^{-1}r}(\xi_0)\cap G)\subset (C_{B^\varepsilon_r(\xi_0)}\backslash \mathcal{B}_n) \subset ( \mathcal{B}^\varepsilon_{ c r}(\xi_0)\cap G).$$ In addition, $n\to \infty$ with $r\to 0$.
\end{lemma}

\begin{proof}
First we show the second inclusion. Let $g\in \mathcal{C}_{B^\varepsilon_r(\xi_0)}$ and observe that it follows by statement (2) of proposition \ref{shadow} and the definition of a cone that $\mathcal{C}_{B^\varepsilon_r(\xi_0)}\cap \mathcal{B}_n= \varnothing$ whenever $C_2e^{-\varepsilon n}\geq r$, and hence $C_{B^\varepsilon_r(\xi_0)}\backslash \mathcal{B}_n =C_{B^\varepsilon_r(\xi_0)}$. Let $n$ be the largest integer that satisfies this inequality. For small enough $r$ we may obviously assume that $C_2e^{-\varepsilon n}\leq 2r$ and it follows by the definition of the visual metric for every $g\in C_{B^\varepsilon_r(\xi_0)}$ that $$d_\varepsilon(g, \xi_0)\leq d_\varepsilon(g, B^\varepsilon_r(\xi_0))+r\leq \varepsilon^{-1}e^{-\varepsilon n_0}+r\leq (2C_2^{-1}\varepsilon^{-1} +1)r.$$

For the first inclusion, let $g\in \mathcal{B}^\varepsilon_{ \tilde r}(\xi_0)\cap G$. Then $\varepsilon^{-1} e^{-\varepsilon|g|}\leq \tilde r$, so it makes sense to define $n$ to be the smallest integer, such that $\varepsilon^{-1} e^{-\varepsilon n}\leq \tilde r$. 
Let $\bar r:=(1+C_4^{-1}\varepsilon)\tilde r$. Then it follows by statement (3) of proposition \ref{shadow} that $S(g,R)\subset B^\varepsilon_{C_4^{-1}\varepsilon \tilde r}(\xi)\subset B^\varepsilon_{\bar r}(\xi_0)$ for every $g\in \mathcal{B}^\varepsilon_{ \tilde r}(\xi_0)\cap G$ and any $\xi \in S(g,R)$. This implies that $g\in \mathcal{C}_{B^\varepsilon_{\bar r}(\xi_0)}\backslash \mathcal{B}_n$.


\end{proof}




\subsubsection{Growth of cones}

In this section we prove that cones grow exponentially, which can be seen as a refinement of (\ref{cannon_improved}). In fact, apart from the proof of corollary \ref{thm_transition_sets}, we do not use the results from this section in the rest of this paper. Nevertheless, we find them informative and important for the intuition. Furthermore, with proposition \ref{shadow} it is not difficult to see that exponential growth of cones implies exponential growth of shadows, which was first proved by Arzhantseva and Lysenok in \cite{arzhantseva} by a different approach.

In definition \ref{cone} we associate to every $g\in \mathcal{C}_{B^\varepsilon_r(\xi_0)}$ the set of points $S(g,R)\subset B^\varepsilon_r(\xi_0)$. On the other hand, we shall find it useful to associate to every $\xi \in B^\varepsilon_r(\xi_0)$ the set  $$\mathcal{U}_\xi:=\{g \in G \ | \ d(g,\xi_{\id})\leq R\},$$ i.e. the set of such points $ g \in G$ that are close to $\xi_{\id}$.

\begin{lemma}\label{estimate_cone}
Let $C_4$ be as in statement (3) of proposition \ref{shadow}. If $g\in \mathcal{U}_\xi$ for some $\xi\in B^\varepsilon_r(\xi_0)$ with $d_\varepsilon(\xi, \xi_0)< r-C_4 e^{-\varepsilon|g|}$, then $g\in \mathcal{C}_{B^\varepsilon_r(\xi_0)}$.
\end{lemma}

\begin{proof}
By statement (3) in proposition \ref{shadow} $S(g,R)\subset B_{C_4 e^{-\varepsilon |g|}}^\varepsilon(\xi)$ for every $\xi \in S(g,R)$. So if $g\in \mathcal{U}_\xi$ for a $\xi$ with $d_\varepsilon (\xi, \xi_0)< r- C_4 e^{-\varepsilon|g|}$, then $$S(g,R)\subset B_{C_4 e^{-\varepsilon |g|}}^\varepsilon(\xi) \subset B^\varepsilon_{r}(\xi_0)$$ and so $g\in \mathcal{C}_{B^\varepsilon_r(\xi_0)}$. 
\end{proof}
 
The next simple fact about coverings will turn out to be helpful when discussing the growth rates of cones. In the following discussion we denote by $I$ a finite index set.

\begin{lemma}\label{covering}
Let $B\subset \p G$ be a set and $r>0$. Let $\{\xi^i\}_{i\in I}\subset B$ denote a maximal set of points, such that $d_\varepsilon(\xi^i,\xi^j)\geq 2r$ for all $i\neq j$. Then $B\subset \bigcup_{i\in I}B^\varepsilon_{2r}(\xi^i)$.
\end{lemma}

\begin{proof}
If not, there exists an $\eta \in B$ such that $d_\varepsilon(\eta,\xi^i)\geq 2r$ for all $i\in I$, which contradicts the maximality of $\{\xi^i\}_{i\in I}$.
\end{proof}

The following proposition shows that cones grow exponentially.

\begin{proposition}[Growth of cones]\label{goc}
There exists a constant $C_5$, such that for every $\xi_0\in \p G$, $r>0$ and $n\in \N$ satisfying $e^{\varepsilon n}\geq C_4/r$ the following estimate holds $$C_5^{-1}r^De^{\varepsilon n D} \leq \# ({\mathcal{C}}_{B^\varepsilon_r(\xi_0)}\cap \mathcal{S}_n)\leq C_5r^De^{\varepsilon n D} .$$
\end{proposition}

\begin{proof}
Let $s_n:=C_4 e^{-\varepsilon n}$ and $N_0\in \N$ be such that  $4s_{N_0}\leq r$. For all $n\geq N_0$, let $\{\xi_i\}_{i\in I}\subset B^\varepsilon_{(r-2s_n)}(\xi_0)$ be a maximal collection of points such that $d_\varepsilon(\xi_i, \xi_j)\geq 2 s_n$ for every $i\neq j$. By lemma \ref{covering}, $$\bigsqcup_{i\in I}B^\varepsilon_{s_n}(\xi_i)\subset B^\varepsilon_{(r-s_n)}(\xi_0) \ \text{ and }\ B^\varepsilon_{(r-2s_n)}(\xi_0)\subset \bigcup_{i \in I}B^\varepsilon_{2s_n}(\xi_i).$$
It then easily follows from statement (2) of proposition \ref{shadow} that \begin{equation}\label{goc1}C_3^{-2}\left(\frac{r}{4}\right)^Ds_n^{-D}\leq \# I\leq C_3^{2}r^Ds_n^{-D}.\end{equation}

On one hand, $\# I$ is the number of $s_n$-balls that fit into $B^\varepsilon_{(r-s_n)}$ without intersecting and can be used to estimate $\# ({\mathcal{C}}_{B^\varepsilon_r(\xi_0)}\cap \mathcal{S}_n)$ from below. Indeed, for $\xi, \tilde \xi \in B^\varepsilon_{(r-2s_n)}(\xi_0)$ with $d_\varepsilon(\xi,\tilde \xi)\geq 2s_n$, there are points $g\in \xi_{\id}$ and $\tilde g \in \tilde \xi_{\id}$ with $g, \tilde g \in ({\mathcal{C}}_{B^\varepsilon_r(\xi_0)}\cap \mathcal{S}_n)$, which define shadows such that $\xi \in S(g,R)$ and $\tilde \xi \in S(\tilde g,R)$. By lemma \ref{estimate_cone}, these shadows are contained in $B^\varepsilon_{(r-s_n)}(\xi_0)$ and by proposition \ref{shadow} $S(g,R)\cap S(\tilde g, R)= \varnothing$. In particular, $g\neq \tilde g$. Hence, for all $\xi^i$ with $i\in I$ there exists a $g_i\in \xi^i_{\id}\cap \mathcal{S}_n$ such that $g_i\neq g_j$ for $i\neq j$, so $$C_3^{-2}\left(\frac{r}{4}\right)^Ds_n^{-D} \leq \#I\leq \# ({\mathcal{C}}_{B^\varepsilon_r(\xi_0)}\cap \mathcal{S}_n).$$ 

On the other hand, we can use the fact that balls $B_{2s_n}^\varepsilon(\xi^i)$ cover $B^\varepsilon_{(r-2s_n)}(\xi_0)$, which gives an upper bound on the number of points in $S_{n-N_1}$, for a large enough $N_1$. To prove this, let $N_1$ be such that $e^{\varepsilon N_1}>2C_2^{-1}C_4$ and choose for every $i\in I$ a point $g_i\in \xi^i_{\id} \cap S_{n-N_1}$. It follows by statement (2) of proposition \ref{shadow} that $$B_{2s_n}^\varepsilon(\xi^i)\subset B_{C_2e^{-\varepsilon (n-N_1)}}^\varepsilon(\xi^i) \subset S(g_i,R),$$ which implies that the shadows $S(g_i,R)$ also cover $B^\varepsilon_{(r-2s_n)}(\xi_0)$. The set $({\mathcal{C}}_{B^\varepsilon_r(\xi_0)}\cap S_{n-N_1})$ is covered by $\mathcal{B}_{2R}(g_i)$, because every point $g\in ({\mathcal{C}}_{B^\varepsilon_r(\xi_0)}\cap S_{n-N_1})$ is at most within distance $R$ of a geodesic ray. So, using (\ref{goc1}), it follows that $$\#({\mathcal{C}}_{B^\varepsilon_{(r-s_n)}(\xi_0)}\cap S_{n-N_1}) \leq \#\mathcal{B}_{2R}(\id)\cdot \# I \leq \#\mathcal{B}_{2R}(\id)\cdot C_3^{-2}r^Ds_n^{-D}.$$ Since $n$ was chosen large enough so that $r\geq s_n$ and by the definition of $N_1$, it follows that $$ \#({\mathcal{C}}_{B^\varepsilon_r(\xi_0)}\cap \mathcal{S}_n)\leq  \#\mathcal{B}_{2R}(\id)\cdot C_3^{-2}(2r)^De^{N_1D}s_n^{-D}.$$
\end{proof}

\subsubsection{Boundaries of cones} \label{sep_sets}

We prove a somewhat technical statement about estimating the boundaries of cones with annuli at infinity. Let $r>0$ be small enough, so that $\p G\backslash {B^\varepsilon_r(\xi_0)}\neq \varnothing$. 

\begin{definition}For $0<r\leq \bar r$ we call a set $\mathcal{A}\subset G \backslash \mathcal{B}_n$ a {\it separating set for $\mathcal{C}_{B^\varepsilon_r(\xi_0)}$ and $G\backslash \mathcal{C}_{B^\varepsilon_{\bar r}(\xi_0)}$ outside $\mathcal{B}_n$}, if the following conditions are met. For every point $g\in \mathcal{A}$ and all $s\in S$ it holds that $gs\notin \mathcal{C}_{B^\varepsilon_r(\xi_0)}\cup (G\backslash \mathcal{C}_{B^\varepsilon_{\bar r}(\xi_0)})$.  
Moreover, every path $p_{g,\tilde g}\subset \mathcal{K}\backslash \mathcal{B}_n$ with $g\in \mathcal{C}_{B^\varepsilon_r(\xi_0)}$ and $\tilde g\in G \backslash \mathcal{C}_{B^\varepsilon_{\bar r}(\xi_0)}$ intersects $\mathcal{A}$.
\end{definition}

Separating sets may be constructed from annuli $ A_r^t(\xi_0):=\overline{(B_{r+t}(\xi_0)\backslash B_{r-t}(\xi_0))}\subset \p G$. 

\begin{lemma}\label{separating_sets}
Let us define for every $n\in \N$ and $t_n:= \min\{e^\varepsilon C_4,e^{\varepsilon(2R+\tilde \delta)}C_2\}\cdot e^{-\varepsilon n}$ the set $$\mathcal{A}_{r,t_n}:=\{g\in \mathcal{U}_\xi \backslash \mathcal{B}_n\ | \ \xi\in A_{r+2t_n}^{t_n}(\xi_0)\}.$$ Then $\mathcal{A}_{r,t_n}$ is a separating set for $\mathcal{C}_{B^\varepsilon_r(\xi_0)}$ and $G\backslash \mathcal{C}_{B^\varepsilon_{r+4t_n}(\xi_0)}$ outside $\mathcal{B}_n$.
\end{lemma}

\begin{proof}
Note that if $\xi\notin B_r^\varepsilon(\xi_0)$ and $g\in \mathcal{U}_\xi$, then $S(g,R)\nsubseteq B_r^\varepsilon(\xi_0)$, so $g\notin \mathcal{C}_{B^\varepsilon_{r}(\xi_0)}$ and similarly for $G\backslash \mathcal{C}_{B^\varepsilon_{r+4t_n}(\xi_0)}$. Moreover, if $\xi\in A_{r+2t_n}^{t_n}(\xi_0)$ it follows by the choice of the constant $t_n$ and by statement (3) of proposition \ref{shadow} that $gs\notin \mathcal{C}_{B^\varepsilon_r(\xi_0)}\cup (G\backslash \mathcal{C}_{B^\varepsilon_{\bar r}(\xi_0)})$ for any $s\in S$. 

Next, we need to show that the set $\mathcal{A}_{r,t_n}$ separates path-connected sets. Since for every $g\in G$ with $|g|> n$ and $s\in S$ the shadows $S(g,R)$ and $S(gs,R)$ are nonempty, there exist rays $\xi,\eta \in \p G$ with $g\in \mathcal{U}_\xi$ and $gs \in \mathcal{U}_\eta$. Because $d(g,gs)=1$, it follows by hyperbolicity that $d(\xi\cap \mathcal{S}_n,\eta\cap \mathcal{S}_n)\leq 2R+\tilde \delta$ and, furthermore, that $d(\xi\cap S_{n-2R-\tilde \delta},\eta\cap S_{n-2R-\tilde \delta})\leq \tilde \delta$, so there exists a $\bar g\in G$ with $|\bar g|\geq n-2R-\tilde \delta$ and $\xi,\eta\in S(\bar g,R)$. It follows from statement (2) of proposition \ref{shadow} that $d^\varepsilon(\xi,\eta)\leq e^{\varepsilon(2R+\tilde \delta)}C_2e^{-\varepsilon n}$. 

Now since $A_{r+2t_n}^{t_n}(\xi_0)$ separates $\p G$ into disjoint sets $B_{r}^\varepsilon(\xi_0)$ and $\p G\backslash B_{r+4t_n}^\varepsilon(\xi_0)$ and because $e^{\varepsilon(2R+\tilde \delta)}C_2e^{-\varepsilon n} < 2t_n$, all paths $p_{g,\tilde g}\subset \mathcal{K}\backslash \mathcal{B}_n$ from $g\in \mathcal{C}_{B^\varepsilon_r(\xi_0)}$ to $\tilde g\in G\backslash \mathcal{C}_{B^\varepsilon_{\bar r}(\xi_0)}$, have to intersect $\mathcal{A}_{r,t_n}$.
\end{proof}

This completes the assembly of tools from geometric group theory that we need.

\section{The variational problem}\label{var_prob}

As explained in the introduction, we are interested in solutions $x:G\to\R$ of the discrete Allen-Cahn equation $$ \rho(\Delta x)_g-V'(x_g)=0 , \text{ for all } \ g\in G ,$$ where $V:\R\to \R$ is a double-well potential and $\rho\geq 0$ is a small constant. More precisely, we assume that $V$ is a Morse function with two absolute minima $c_0$ and $c_1$, i.e. $V(c_0)=V(c_1)\leq V(s)$ for all $s\in \R$. Moreover, we assume that $c_0$ and $c_1$ are the only absolute minima of $V$ in the interval $[c_0,c_1]$.

The equation (\ref{rr}) comes with a variational structure, which we explain below. It follows from theorem \ref{ac_theorem} that the equation (\ref{rr}) has many solutions, but we are interested only in solutions that minimise the action globally.

\subsection{Minimal solutions}

The variational structure that the problem (\ref{rr}) carries is the following. For any compact (i.e. finite) set $\mathcal{B} \subset G$ and function $x:G\to \R$, we define the ``action functional'' \begin{equation}\label{vp}W^\rho_\mathcal{B}(x):=\sum_{g\in \mathcal{B}}\sum_{s\in S}\left(\frac{\rho}{4}(x_{gs}-x_g)^2+V(x_g)\right).\end{equation} It is easy to see that for every $\rho>0$ the function $W^\rho_\mathcal{B}(x)$ is a function of variables $x_g$ where $g\in \mathcal{B}^{\out}$ and that $x$ is a solution of (\ref{rr}), if for every compact $\mathcal{B}\subset G$ and every perturbation $v$ supported on $\mathcal{B}^{\inn}$, $$\left.\frac{d}{ds}\right|_{s=0} W^\rho_B(x+sv)=0.$$ 
This motivates the following definition.

\begin{definition}\label{def_gm}For any set $\mathcal{B}\subset G$, a function $x:G \to \R$ is called {\it a minimiser on $\mathcal{B}$}, if $$ W^\rho_\mathcal{B}(x+v)-W^\rho_{\mathcal{B}}(x)\geq 0$$ for all $v:G \to \R$ with compact support in ${\mathcal{B}}^{\inn}$. 

This definition makes sense also for infinite sets $\mathcal{B}$, because the support of $v$ is compact, and one may evaluate the difference above by truncating the actions to a bounded set containing the support of $v$ in its interior.
The function $x$ is called a {\it global minimiser}, if $x$ is a minimiser on $G$.
\end{definition}

Observe that with this definition, a minimiser $x$ on a compact set $\mathcal{B}$ is a solution to the Dirichlet problem given by (\ref{rr}) on ${\mathcal{B}^{\inn}}$, with given boundary values $x|_{\p^{\f}\mathcal{B}}$ on $\p^{\f}\mathcal{B}=\p^{\out}\mathcal{B}\cup \p^{\inn}\mathcal{B}=\mathcal{B}^{\out}\backslash \mathcal{B}^{\inn}$.

\begin{remark}\label{xc01}It is easy to see that the constant function $x^{c_0}\equiv c_0$ and $x^{c_1}\equiv c_1$ are global minimisers of (\ref{vp}).
\end{remark}

Obviously, all global minimisers are solutions and it is clear that one may construct global minimisers as limits of minimisers on compact domains that exhaust $G$. More precisely, if $\mathcal{B}^n$ is a sequence of compact subsets exhausting $G$ and $x^n$ a sequence of minimisers on $\mathcal{B}^n$ which converges to a function $x^\infty$, then $x^\infty$ is a global minimiser.

The following statements are standard for elliptic difference operators and we provide them for the sake of completeness (see also \cite{candel-llave} or \cite{llave-lattices}).

\begin{lemma}\label{mm_lemma}
For functions $x,y$, their point-wise minimum and maximum $m:=\min\{x,y\}$ and $M:=\max\{x,y\}$, and for any compact domain $\mathcal{B}\subset G$ it holds that $$W^\rho_{\mathcal{B}}(x)+W^\rho_{\mathcal{B}}(y)\geq W^\rho_{\mathcal{B}}(M)+W^\rho_{\mathcal{B}}(m).$$
\end{lemma}

\begin{proof}
Write $\alpha:=M-x$ and $\beta:=m-x$ and observe that $\alpha \geq0$, $\beta\leq0$, while $\mbox{supp}(\alpha) \cap \mbox{supp}(\beta) = \emptyset$ and $y=M+m-x=\alpha + m=\alpha + \beta + x$. Rewriting the inequality above as $$W^\rho_{\mathcal{B}}(x+\alpha+\beta)-W^\rho_{\mathcal{B}}(x+\alpha)-W^\rho_{\mathcal{B}}(x+\beta)+W^\rho_{\mathcal{B}}(x) \geq 0,$$
we can write its left-hand side in an integral form 
 $$\int_0^1\int_0^1\frac{\p^2}{\p t\p s} W^\rho_{\mathcal{B}}(x+\alpha t+\beta \tilde t)\,d\tilde t\, dt =  \int_0^1\int_0^1\sum_{g,\tilde g\in G}\frac{\p^2}{\p_g\p_{\tilde g}} W^\rho_{\mathcal{B}}(x+\alpha t+\beta \tilde t)\alpha_g\beta_{\tilde g} \ d\tilde t \,dt\ .$$
Since $\mbox{supp}(\alpha) \cap \mbox{supp}(\beta) = \emptyset$, we have that $\alpha_g\beta_g=0$ for all $g$, and hence only mixed derivatives with $\tilde g=gs$ remain. Since $$\frac{\p^2}{\p_g\p_{gs}} W^\rho_{\mathcal{B}}(x)\leq -\frac{\rho}{2}$$ for all $g\in {\mathcal{B}^{\inn}}$ and $\alpha_g \beta_{gs}\leq 0$, the claim follows.
\end{proof}

\begin{lemma}\label{comparison}
Let $x,y$ be two minimisers on $\mathcal{B}$ with $x_g\leq y_g$ for all $g\in \p^{\f}\mathcal{B}$. Then either $x_g<y_g$ for all $g\in \mathcal{B}^{\inn}$, or $x|_{\mathcal{B}}\equiv y|_{\mathcal{B}}$.

In particular, any two global minimisers $x,y$ with $x_g\leq y_g$ for all $g\in G$ and such that $x\neq y$, are totally ordered: $x_g<y_g$ for all $g\in G$.
\end{lemma}

\begin{proof}
Define the minimum $m$ and maximum $M$ of $x$ and $y$ and note that $x|_{\p^{\f}\mathcal{B}}=m|_{\p^{\f}\mathcal{B}}$ and that $y|_{\p^{\text{f}}\mathcal{B}}=M|_{\p^{\f}\mathcal{B}}$. It follows by definition of a minimiser and by lemma \ref{mm_lemma} that $m$ and $M$ are also minimisers on $\mathcal{B}$. Assume that $g\in {\mathcal{B}^{\inn}}$ such that $x_g=y_g$. Then $\Delta_g(M)=\Delta_g(m)$ and by the maximum principle $M_{gs}-m_{gs}=0$ for all $s\in S$. Inductively it then follows that $x|_{\mathcal{B}}\equiv y|_{\mathcal{B}}$. 
\end{proof}

Existence of minimisers for compact domains is a well known fact that follows from coercivity of the action (\ref{vp}). However, we shall construct global minimisers as solutions from the so-called anti-integrable limit in the next section, so we leave out the proof of the following informative lemma.

\begin{lemma}\label{existence_lemma}Let $\mathcal{B}\subset G$ be a compact set and let $f:\p^{\f}\mathcal{B} \to \R$ be given. Then there exists a minimiser $x$ of (\ref{vp}) on $\mathcal{B}$, with boundary values given by $f$, that is $x|_{\p^{\f}\mathcal{B}}=f$.
\end{lemma}

%

\subsection{The anti-continuum limit}\label{anticontinuum}
In this section we explain that for small enough constants $\rho$ in the equation (\ref{rr}), a wealth of solutions may be obtained by a version of the implicit function theorem. This very useful method has been used extensively in Aubry-Mather theory, see e.g. \cite{abramovici,baesens,MackayMeiss,anticontinuum}.

In the case that $\rho=0$, the equation (\ref{rr}) reads \begin{equation}\label{ah} V'(x_g)=0 , \text{ for all } \ g\in G , \end{equation} which is solved by requiring that for every $g\in G$, $x_g$ is a critical point of $V$. A proof based on Newton iteration shows that it is possible to find solutions of (\ref{rr}) for small constants $\rho$, as continuations of the known solutions of the problem (\ref{ah}). This is the content of the following theorem, the proof of which is very similar to those in \cite{MackayMeiss,anticontinuum}. We provide a proof for the reader's convenience in the appendix.

\begin{theorem}\label{ac_theorem}
For every solution $x$ of the anti-continuum limit problem (\ref{ah}) with $x_g\in [c_0,c_1]$ for all $g\in G$, there exist constants $\sigma_0>0$ and $\rho_0>0$, independent of $x$, such that for every $\rho$ with $0\leq \rho\leq \rho_0$ and every set $\mathcal{B}\subset G$, there exists a unique function $x^\rho:G\to \R$ with $\|x^\rho-x\|_{\infty}\leq \sigma_0$ which solves (\ref{rr}) on $\mathcal{B}^{\inn}$  and coincides with $x$ on $G\backslash \mathcal{B}^{\inn}$.
Moreover, $\|x^\rho-x\|_{\infty}\to 0$ as $\rho\to 0$, so that we may write $x=x^0$.
\end{theorem}

Next, we show that for small enough $\rho$ all uniformly bounded solutions with values in $[c_0-\sigma_0,c_1+\sigma_0]$ can be found as continuations from the anti-continuum limit and that minimisers are continuations of solutions which have values in absolute minima of $V$. 

\begin{theorem}\label{gm_continuations}
Let $\rho_0$ and $\sigma_0$ be as obtained in theorem \ref{ac_theorem}. Then there exists a $\rho_1$ with $0< \rho_1\leq \rho_0$, such that for all $0<\rho\leq \rho_1$ the following holds. For any function $\tilde x:G\to \R$ that solves (\ref{rr}) with such a $\rho$ on some set $\mathcal{B}^{\inn}\subset G$ and satisfies $\tilde x_g\in \{c_0,c_1\}$ for all $g\in G\backslash \mathcal{B}^{\inn}$, there exists a solution $x$ of the anti-continuum problem (\ref{ah}), such that $\|x-\tilde x\|_\infty\leq \sigma_0$. In other words, with the notation from theorem \ref{ac_theorem} we may write $\tilde x=x^\rho$ for a solutions $x=x^0$ of the anti-continuum limit (\ref{ah}).

Furthermore, for a small enough constant $\rho_1$, $x^0_g\in \{c_0,c_1\}$ for all $g\in G$, whenever $\tilde x = x^\rho$ is a minimiser on $\mathcal{B}$. 
\end{theorem}

\begin{proof}
Because $V$ is Morse, there exists for any $\sigma>0$ a $\rho>0$ such that if $|V'(X)|<\rho$ then $|X - c| < \sigma$ for some critical point $c$ of $V$. Now suppose that $\tilde x$ is a solution to (\ref{rr}) as stated in the theorem. By the proof of theorem \ref{gm_continuations}, $\sigma_0\leq (1/2)(c_1-c_0)$, and so for all $g\in G$ and $s\in S$ it follows that $|\tilde x_g- \tilde x_{gs}|\leq 2(c_1-c_0)$. This implies that
 $$|V'(\tilde x_g)|=|\mathbb{1}_{\mathcal{B}^{\inn}} (\rho \Delta_g(\tilde x))|\leq 2\rho(c_1-c_0)(\#S)\ \mbox{uniformly for}\ g\in G\, ,$$ where $\mathbb{1}_{\mathcal{B}}$ denotes the indicator function on the set ${\mathcal{B}}\subset G$.
Thus, there is a $0<\rho_1\leq \rho_0$ so that if $0\leq \rho\leq \rho_1$, then for any $g\in G$ there is a critical point $c\in[c_0,c_1]$ of $V$, such that $|\tilde x_g-c|< \sigma_0$ for all $g$. In other words, for all such $\rho$ there exists a solution $x$ of (\ref{ah}), such that $\|\tilde x-x\|_{\infty}\leq \sigma_0$. Because $x^\rho$ is the unique solution to (\ref{rr}) on $\mathcal{B}$ with $\|x^\rho-x\|_{\infty}\leq \sigma_0$, this implies that $x^{\rho}=\tilde x$.

To prove the second part of the theorem observe that, since $c_0$ and $c_1$ are the only absolute minima of $V$ in the interval $[c_0,c_1]$ and $V$ is Morse, there is a constant $\tilde d$ such that $V(c)-V(c_0)\geq \tilde d$ for every critical point $c \in (c_0, c_1)$. Assume that $x^0_{g_0}=c$, for some critical point $c\in(c_0,c_1)$ and some $g_0\in G$. Taking $\rho_1$ small enough, it follows from the continuity of $V$ that $V(x^\rho_{g_0})\geq \frac{\tilde d}{2}+V(c_0)$ for all $\rho\leq \rho_1$. Furthermore, by taking $\rho_1$ even smaller if necessary, we may assume that $(\#S)(c_1-c_0)^2 \rho_1\leq 2\tilde d.$ Defining a variation $\bar x$ of $x^\rho$ by $\bar x_g=x^\rho_g$ for all $g\neq g_0$ and $\bar x_{g_0}=c_0$, it follows that $$W_{\{{g_0}\}}(x^\rho)=\sum_{s\in S}\frac{\rho}{4}(x_{{g_0}s}-x_{g_0})^2+V(x_{g_0})\geq \frac{\tilde d}{2}+V(c_0)\geq (\#S)\frac{\rho}{4}(c_1-c_0)^2+V(c_0)\geq W_{\{{g_0}\}}(\bar x),$$ so $x^\rho$ is not a global minimiser.

\end{proof}

The following corollary is a direct consequence of theorem \ref{gm_continuations} together with lemma \ref{comparison}, applied to the global minimisers of (\ref{vp}) that are constant with values either $c_0$ or $c_1$ (see remark \ref{xc01}).

\begin{corollary}\label{gm_range}Let $\rho\leq \rho_1$ and $\tilde x=x^\rho$ be, as in the second part of theorem \ref{gm_continuations}, a minimiser on $\mathcal{B}\subset G$. Then, for all $g\in G$, $x^\rho_g\in[c_0,c_0+\sigma_0)\cup(c_1-\sigma_0,c_1]$.
\end{corollary}

To analyse the behaviour of a minimiser, we may now analyse the set of points $g\in G$ where $x^0_g-x_{gs}^0\neq 0$. This gives us an estimate on the action of the $x^\rho$, because such points give an action contribution of size $\rho$ and the rest of the terms have a positive action contribution.

\section{Minimal Dirichlet problem at infinity}\label{DP}

We wish to construct global minimisers that solve the Dirichlet problem at infinity as given in definition \ref{dp_def}. In other words, we are looking for global minimisers which are on neighbourhoods near the boundary at infinity uniformly close to either $c_0$ or $c_1$.
%
We shall construct them as limits of minimisers on balls with growing radii, using the concept of the cone from definition \ref{cone}.

\begin{proposition}\label{local_minimisers}
Let $c_0< c_1$ be the two distinct absolute minima of $V$ and let $D_0, D_1\subset \p G$ be as in definition \ref{dp_def}.
Define a solution $\tilde x$ of the anti-continuum limit (\ref{ah}) by $$\displaystyle \tilde x_g:=\begin{cases}c_0 & \text{ if } g\in (\mathcal{C}_{D_0})^{\inn} ,\\ c_1 & \text{ else.} \end{cases}$$
For every $N\in \N$, let $x^N:G\to \R$ be a minimiser on $\mathcal{B}_N$, such that $x^N|_{((\mathcal{B}_N)^{\inn})^c}\equiv \tilde x|_{((\mathcal{B}_N)^{\inn})^c}$.
Then, as $N\to \infty$, $x^N$ converge along a subsequence to a global minimiser $\bar x$ of (\ref{rr}).
\end{proposition}

\begin{proof}
By lemma \ref{comparison} $c_0\leq x^N_g\leq c_1$ for all $N\in \N$ and all $g\in G$. By Tychonov's theorem $[c_0,c_1]^{G}$ is a compact space w.r.t. pointwise convergence, so $x^N$ has a convergent subsequence $x^{N_k}\to \bar x$, which is then a global minimiser.

\end{proof}


\subsection{Transition sets}
Let $x^N$ be as in proposition \ref{local_minimisers}. By corollary \ref{gm_range} the domain of the function $x^N$ in $\mathcal{B}_{N+1}$ can be split in the following two sets $$ \mathcal{B}_N^{c_0}:=\{ g \in \mathcal{B}_{N+1} \ | \ \tilde x_g\in [c_0,c_0+ \sigma_0)\}\ \text{ and }\ \mathcal{B}_N^{c_1}:=\{ g \in \mathcal{B}_{N+1} \ | \ \tilde x_g\in (c_1- \sigma_0,c_1]\}.$$ 
We define  the ``transition set'', as  $$\mathcal{T}_N(x^N):=\{g\in \mathcal{B}_{N+1} \ | \ |x^N_g-x^N_{gs}|\geq 2\sigma_0\}=\p^{\f}\mathcal{B}_N^{c_0}=\p^{\f} \mathcal{B}_N^{c_1}.$$  

We shall show that, as a consequence of minimality, the set $\mathcal{T}_N(x^N)$ in some sense cannot be too large, and so the sets $\mathcal{B}_N^{c_{0}}$ and $\mathcal{B}_N^{c_{1}}$ behave somewhat nicely. The first statement in this direction is about connectedness of the sets $\mathcal{B}_N^{c_0}$ and $\mathcal{B}_N^{c_1}$, where a set $\mathcal{D}\subset G$ is connected if the corresponding set of points together with the edges connecting them is connected in the Cayley graph $\mathcal{K}$. The following lemma states that every connected component of the sets $\mathcal{B}_N^{c_0}$ and $\mathcal{B}_N^{c_1}$ needs to touch the boundary of the ball $\mathcal{B}_N$.

\begin{lemma}\label{connected_components}For every connected component $\mathcal{D}^0$ of $\mathcal{B}_N^{c_0}$ and $\mathcal{D}^1$ of $\mathcal{B}_N^{c_1}$, $$\mathcal{D}^{0} \cap \p^{\out} \mathcal{B}_N\neq \varnothing \ \text{ and }  \mathcal{D}^{1} \cap \p^{\out} \mathcal{B}_N\neq \varnothing \ .$$ 
\end{lemma}
\begin{proof}
Assume that there is a connected component $\mathcal{D}^0\subset \mathcal{B}_N^{c_0}$ which does not intersect the boundary $\p^{\out} \mathcal{B}_N$. Then $x^N_g\in [c_1-\sigma_0,c_1]$ on $\p^{\out}\mathcal{D}^0=(\mathcal{D}^0)^{\out} \backslash \mathcal{D}^0$ and $x^N_g\in [c_0,c_0+\sigma_0]$ on $ \mathcal{D}^0$. Let $$ S^\rho_g(x):=\sum_{s\in S}\left(\frac{\rho}{4}(x_{gs}-x_g)^2+V(x_g)\right), $$ and define the variation $\tilde x^N$ of $x^N$ supported on $\mathcal{D}^0$ by $\tilde{x}^N\equiv c_1$ on $ \mathcal{D}^0$ and by $\tilde{x}^N\equiv x^N$ on $G\backslash\mathcal{D}^0$. Obviously $(S^\rho_g(\tilde x^N)-S^\rho_g(x^N))\leq 0$ for all $g\in (\mathcal{D}^0)^{\inn}$ and it is easy to see that for all $g\in \p^{\f}\mathcal{D}^0$, $$S_g^\rho(\tilde x^N)-S_g^\rho(x^N)\leq \frac{\rho}{4}((\#S)\sigma_0^2-(c_1-c_0-2\sigma_0)^2). $$ 
If necessary, we further reduce $\rho_1$ defined as in section \ref{anticontinuum}, so that $\sigma_0<(c_1-c_0)/3$. By the inequality in theorem (\ref{ac_theorem}) it then follows that 
$$\begin{aligned}W_{(\mathcal{D}^0)^{\out}}(\tilde x^N)- W_{(\mathcal{D}^0)^{\out}}(x^N) &\leq \sum_{g\in (\mathcal{D}^0)^{\inn}}(S^\rho_g(\tilde x^N)-S^\rho_g(x^N))+\sum_{g\in \p^{\f}\mathcal{D}^0}(S^\rho_g(\tilde x^N)-S^\rho_g(x^N))
\\ &
\leq \frac{\rho}{4}((\#S)\sigma_0^2-(c_1-c_0-2\sigma_0)^2) \# (\p^{\f}\mathcal{D}^0)<0,\end{aligned}$$ 
which is a contradiction to the fact that $x^N$ is a minimiser.



\end{proof}

The next lemma is about estimating the size of the boundary $\p^{\out} \mathcal{B}_N^{c_1}$ within a domain $\mathcal{D}$ by the size of the boundary of $\mathcal{D}$ which is within $\mathcal{B}_N^{c_1}$. The estimate we obtain can be seen as a weak quasi-minimality-type condition on the transition set $\mathcal{T}_N(x^N)$, in the spirit of the Morse lemma for quasi-geodesics. 

\begin{lemma}\label{ts_estimate}
For every finite set $\mathcal{D}\subset G$ and every $x^N$ defined as in proposition \ref{local_minimisers}, $$ \#((\p^{\out} \mathcal{B}_N^{c_1}\cap \mathcal{D})\leq 6(\# S) \# (\p^{\inn} \mathcal{D}\cap \mathcal{\mathcal{B}}_N^{c_1}).$$
\end{lemma}

\begin{proof}
Define a variation $\tilde x^N$ of $x^N$ with support in $\mathcal{B}_N^{c_1}\cap(\mathcal{D}^{\inn})$:$$\tilde x^N:=\begin{cases}c_0 \ & \text{ for all } g \in \mathcal{B}_N^{c_1}\cap (\mathcal{D}^{\inn}) \ \\ x^N \ & \text{ else. } \end{cases}$$
Look at the outer set $\mathcal{C}:= (\mathcal{B}_N^{c_1}\cap(\mathcal{D}^{\inn}) )^{\out}$ and the inner set $\mathcal{I}:=(\mathcal{B}_N^{c_1}\cap(\mathcal{D}^{\inn})^{\inn}$, such that $\mathcal{I}^{\out}=\mathcal{C}^{\inn}$. We will estimate the action functional $W_{\mathcal{C}}^\rho(\tilde x^N)-W_{\mathcal{C}}^\rho(x^N)$. In view of that, we first split the set $\mathcal{C}^{\out}$ into disjoint subsets by $$\mathcal{C}^{\out}= \mathcal{I}\cup (\mathcal{C}^{\inn} \backslash \mathcal{I})\cup (\mathcal{C}\backslash\mathcal{C}^{\inn})\cup  (\mathcal{C}^{\out}\backslash \mathcal{C})=\mathcal{I}\cup \p^{\out}\mathcal{I}\cup\p^{\inn}\mathcal{C}\cup\p^{\out}\mathcal{C}$$ and note that by the definition of $\tilde x^N$ $$W_{\mathcal{I}}^\rho(\tilde x^N)- W_{\mathcal{I}}^\rho( x^N)\leq 0.$$ 

Next, we estimate $W_{\p^{\inn}\mathcal{C}}^\rho(\tilde x^N)-W_{\p^{\inn}\mathcal{C}}^\rho(x^N)$. Obviously $\tilde x^N_g=x^N_g$ for every $g\in \p^{\inn}\mathcal{C}$ and so \begin{equation}\label{ie1}W_{\p^{\inn}\mathcal{C}}^\rho(\tilde x^N)-W_{\p^{\inn}\mathcal{C}}^\rho(x^N)=
\sum_{g\in \p^{\inn}\mathcal{C}}\sum_{\substack{s\in S \\ gs \in  \mathcal{C}^{\inn}}}\frac{\rho}{4}\left((\tilde x^N_{gs}-x^N_g)^2-(x^N_{gs}-x^N_g)^2\right). \end{equation}
We would like to to split the sum in (\ref{ie1}) into sums over $\p^{\inn}\mathcal{C}\cap \mathcal{B}_N^{c_1}$ and $\p^{\inn}\mathcal{C}\cap \mathcal{B}_N^{c_0}$. We recall from definition of $\mathcal{C}$ and equation (\ref{boundary}) that $$\p^{\inn}\mathcal{C}=\p^{\out} (\mathcal{B}_N^{c_1}\cap \mathcal{D}^{\inn})= (\p^{\out} \mathcal{B}_N^{c_1}\cap {\mathcal{D}})\cup ((\mathcal{B}_N^{c_1})^{\out}\cap \p^{\inn} \mathcal{D}).$$ It thus follows by disjointness of $\mathcal{B}_N^{c_{0}}$ and $\mathcal{B}_N^{c_{1}}$ that 
$$ \p^{\inn}\mathcal{C} \cap \mathcal{B}_N^{c_0}=\p^{\out} \mathcal{B}_N^{c_1}\cap {\mathcal{D}} \ \text{ and } \ \p^{\inn}\mathcal{C} \cap \mathcal{B}_N^{c_1}=\mathcal{B}_N^{c_1}\cap \p^{\inn} \mathcal{D}.$$ 
By corollary \ref{gm_range} $(\tilde x^N_{gs}-x^N_g)^2\leq (c_1-c_0)^2$ for all $g\in \p^{\inn}\mathcal{C}\cap \mathcal{B}_N^{c_1}$. Furthermore, it holds for all $g\in \p^{\inn}\mathcal{C}\cap \mathcal{B}_N^{c_0}$ that $(\tilde x^N_{gs}-x^N_g)^2\leq \sigma_0^2$ and $(x^N_{gs}-x^N_g)^2\geq (c_1-c_0-2\sigma_0)^2$ and by the inequlity in theorem \ref{ac_theorem} it follows that $$(c_1-c_0-2\sigma_0)<\frac{c_1-c_0}{2} \ \mbox{ and } \ \frac{(c_1-c_0)^2}{4}-\sigma_0^2>\frac{(c_1-c_0)^2}{6},$$ $$\text{ so } \ (\tilde x^N_{gs}-x^N_g)^2-(x^N_{gs}-x^N_g)^2\leq \frac{1}{6}(c_1-c_0)^2.$$
Thus we obtain $$W_{\p^{\inn}\mathcal{C}}^\rho(\tilde x^N)-W_{\p^{\inn}\mathcal{C}}^\rho(x^N)\leq \#(\mathcal{B}_N^{c_1}\cap \p^{\inn} \mathcal{D})
\frac{\rho(\# S)}{4}(c_0-c_1)^2-\#(\p^{\out} \mathcal{B}_N^{c_1}\cap {\mathcal{D}})
\frac{\rho}{24}(c_0-c_1)^2.$$

To estimate $W_{\p^{\out}\mathcal{I}}^\rho(\tilde x^N)-W_{\p^{\out}\mathcal{I}}^\rho(x^N)$, we observe  that $V(\tilde x_g^N)\leq V(x^N_g)$ for every $g\in \mathcal{C}^{\inn}$. Moreover, if $gs\in \mathcal{C}^{\inn}$ then $(\tilde x^N_{gs}-\tilde x_g^N)=0$ and if $gs\notin \mathcal{C}^{\inn}$, then $\tilde x^N_{gs}= x^N_{gs}$. This implies that \begin{equation}\label{ie2}W_{\p^{\inn}\mathcal{I}}^\rho(\tilde x^N)-W_{\p^{\inn}\mathcal{I}}^\rho(x^N)\leq
\sum_{g\in \p^{\inn}\mathcal{I}}\sum_{\substack{ s\in S\\ gs \notin  \mathcal{C}^{\inn}}}\frac{\rho}{4}\left((x^N_{gs}-\tilde x^N_g)^2-(x^N_{gs}-x^N_g)^2\right).\end{equation}
By symmetry the sum in (\ref{ie1}) is equal to the sum (\ref{ie2}), so 
$$0\leq W_{\mathcal{C}}^\rho(\tilde x^N)-W_{\mathcal{C}}^\rho(x^N)\leq 2\left( \#(\mathcal{B}_N^{c_1}\cap \p^{\inn} \mathcal{D})
\frac{\rho(\# S)}{4}(c_0-c_1)^2-\#(\p^{\out} \mathcal{B}_N^{c_1}\cap {\mathcal{D}})
\frac{\rho}{24}(c_0-c_1)^2\right),$$ which finishes the proof.

\end{proof}

\subsection{Main lemma}

In this section we prove the main technical result about how the transition sets $\mathcal{T}_N(x^N)$ behave when $N$ goes to infinity. It roughly shows that if the transition set for some large enough $N$ is locally too big, than it grows much faster than $\mathcal{B}_n$. This will be used in section \ref{dp_section} to show that uniformly in $N$ the transition sets $\mathcal{T}_N(x^N)$ extend towards the identity. 

For the remainder of this text let us assume, by taking the constant $\varepsilon>0$ from the visual metric sufficiently small enough, that \begin{equation}\label{dim_large}D=\frac{h}{\varepsilon}>\frac{1}{4}.\end{equation}

\begin{definition}\label{rini}

Let us define for $r>0$ and $\xi_0\in \p G$ the following objects:

\begin{itemize}

\item 
the sequence $r_i\to r$ by $r_0:=0$ and \begin{equation}\label{ri}r_i:=\frac{6 r}{\pi^2}\sum_{j=1}^i\frac{1}{j^2}\ \text{ and denote } \ d_i:=r_{i+1}-r_i=\frac{6 r}{\pi^2 (i+1)^2} \ \text{ for all }\ i\geq 1 \ ,\end{equation}  

\item 
for a given natural number $n_1$ the increasing sequence of real numbers \begin{equation}\label{ni}n_{i+1}:=\left(\frac{D+1/2}{D+1/4}\right)n_i=\left(\frac{D+1/2}{D+1/4}\right)^{i-1}n_1\ ,\end{equation} 

\item 
for $r_i$ and $n_i$ as above \begin{equation}\label{vi}\mathcal{V}_i:=(\mathcal{C}_{B^\varepsilon_{r_{i+1}}(\xi_0)}\backslash \mathcal{C}_{B^\varepsilon_{r_i}(\xi_0)})\backslash \mathcal{B}_{\lfloor n_i\rfloor} \,\end{equation} where $\lfloor \cdot \rfloor$ denotes the floor function. 
\end{itemize}

\end{definition}

\begin{remark}\label{vi_structure}
Let $r_i$, $n_i$ and $\mathcal{V}_i$ be as in the definition above. Recall the definition of the constant $t_n$ from lemma \ref{separating_sets} and rewrite it for every $n\in \N$ as  \begin{equation}\label{konstants}t_n= \frac{k_1^{-1}}{4} e^{-\varepsilon n} \ \text{ where } \ k_1^{-1}:=4\min\{e^\varepsilon C_4,e^{\varepsilon(2R+\tilde \delta)}C_2\}.\end{equation} Moreover, recall from section \ref{sep_sets} the notion of a separating set $\mathcal{A}_{r,t_n}$ for $\mathcal{C}_{B^\varepsilon_r(\xi_0)}$ and $G\backslash \mathcal{C}_{B^\varepsilon_r(\xi_0)}$ outside $\mathcal{B}_n$ which is, by lemma \ref{separating_sets}, given by the annulus at infinity $A_{r+2t_n}^{t_n}(\xi_0)$. Using definitions (\ref{vi}) and (\ref{konstants}) we may, for every $i\in \N$, write the set $\mathcal{V}_i$ as a disjoint union $\mathcal{V}_i=\bigsqcup_{j\in I_i}\tilde{\mathcal{A}}_{i_j}$, where the sets $\tilde{\mathcal{A}}_{i_j}$ satisfy $\mathcal{A}_{r_{i_j},t_{\lfloor n_i\rfloor}} \subset \tilde{\mathcal{A}}_{i_j}$ and are thus also separating sets. It follows from definitions of $r_i$ and $t_n$ that \begin{equation}\label{eq1}\# I_i\geq  \left\lfloor \frac{k_1d_i}{ e^{-\varepsilon \lfloor n_i\rfloor}}\right\rfloor-1 \geq k_1 d_i e^{-\varepsilon}e^{\varepsilon n_i}-2.\end{equation} 
\end{remark}

The size of $I_i$ measures the number of separating sets outsize a ball of radius $\lfloor n_i\rfloor$  contained in $\mathcal{V}_i$ and grows exponentially with $n_i$. This is an essential part to the proof of the following lemma, which gives us a powerful estimate about the growth of the set $\mathcal{C}_{B^\varepsilon_{r_i}(\xi_0)}\cap \mathcal{B}_N^{c_1}$. The idea behind the proof is the following. From lemma \ref{ts_estimate} we deduce that if $\mathcal{C}_{B^\varepsilon_{r_i}(\xi_0)}\cap \p^{\out} \mathcal{B}_N^{c_1}$ is large, then for all $j\in I_i$, the set $\p^{\inn} \mathcal{C}_{B^\varepsilon_{r_{i_j}}(\xi_0)}\cap \mathcal{B}_N^{c_1}$ can be at most by a factor smaller. By fitting $\p^{\inn} \mathcal{C}_{B^\varepsilon_{r_{i_j}}(\xi_0)}$ into separating sets $\tilde{\mathcal{A}}_{i_j}$ and by using the isoperimetric inequality, the size of $I_i$ implies that $\mathcal{V}_i\cap \mathcal{B}_N^{c_1}$ is exponentially larger than $\mathcal{C}_{B^\varepsilon_{r_i}(\xi_0)}\cap \mathcal{B}_N^{c_1}$. In particular, also $\mathcal{C}_{B^\varepsilon_{r_{i+1}}(\xi_0)}\cap \mathcal{B}_N^{c_1}$ is exponentially larger and the strategy can be repeated by induction.

\begin{lemma}[Main lemma]\label{main_lemma}
Let $r>0$ and $\xi_0\in \p G$ be given and define a constant $k$ by \begin{equation}\label{k}k:=\left(48 (\#S) k_0 \tilde C\right)^{(\frac{4D}{4D-1})}, \end{equation} where the definition of the constants $\tilde C$ and $k_0$ comes from (\ref{cannon_improved}) and from the isoparametric inequality given in lemma \ref{isoperimetric}, respectively. By (\ref{dim_large}) we may furthermore define $L_0$ as the smallest natural number such that $(\log(L_0))^{4D}<L_0$, we recall the definition of the entropy $h$ from proposition \ref{shadow} and let $k_1$ be as in (\ref{konstants}).

Let  $r_i$, $n_i$ and $\mathcal{V}_i$ be as in definition \ref{rini}, whereby $n_1$ is a real number satisfying
\begin{equation}\label{n1}n_1\geq \max\left\{\frac{4}{\varepsilon (4D+1)}\log\left(\frac{L_0}{k}\right) \ , \ 32 h\ , \ \frac{2}{\varepsilon}\log\left(\frac{(k+2\tilde C)4\pi^2}{3r\tilde C k_1e^{-\varepsilon}}\right)\right\}.\end{equation}

Then, whenever there exist an $n_i$ as in (\ref{ni}) satisfying \begin{equation}\label{ih} \#(\mathcal{C}_{B^\varepsilon_{r_i}(\xi_0)}\cap \mathcal{B}_N^{c_1})\geq ke^{\varepsilon (D+\frac{1}{4}) n_i}\end{equation} for some $i$ such that $n_i<N$, then it holds for all integers $\iota\geq i$ with $n_\iota<N$ that $$\#(\mathcal{V}_{\iota}\cap \mathcal{B}_N^{c_1})\geq ke^{\varepsilon (D+\frac{1}{4}) n_\iota}.$$
\end{lemma}

\begin{proof}
By inclusion and by conditions (\ref{n1}) and (\ref{ih}), $$\#( \mathcal{C}_{B^\varepsilon_{r_{i_j}}(\xi_0)}\cap \mathcal{B}_N^{c_1})\geq ke^{\varepsilon (D+\frac{1}{4}) n_i}\geq ke^{\varepsilon (D+\frac{1}{4}) n_1}\geq L_0.$$ 
By definition of $L_0$ it follows for every $L\geq L_0$ that $$ \frac{L}{\log(L)}> L^{\left(1-\frac{1}{4D}\right)}.$$ 
Hence it holds by the isoperimetric inequality (lemma \ref{isoperimetric}) for any set $\mathcal{D}$ with $\# \mathcal{D}\geq L_0$ that \begin{equation}\label{ve} k_0\#( \p^{\out} \mathcal{D})\geq (\# \mathcal{D})^{\left(1-\frac{1}{4D}\right)}.\end{equation}
Since $4D>1$ by (\ref{dim_large}), it follows from the first condition in (\ref{n1}) and assumption (\ref{ih}) that $$ k_0 \#( \p^{\out} ( \mathcal{C}_{B^\varepsilon_{r_{i_j}}(\xi_0)}\cap \mathcal{B}_N^{c_1}))\geq k^{(1-\frac{1}{4D})}e^{\varepsilon D n_i}, $$
 for all $j\in I_i$ where $I_i$ is as in remark \ref{vi_structure} the index set that gives $r_{i_j}$ and $\tilde{\mathcal{A}}_{i_j}$ such that $\mathcal{V}_i=\bigsqcup_{j\in I_i}\tilde{\mathcal{A}}_{i_j}$. Since $ (\mathcal{B}_N^{c_1})^{\out}=\mathcal{B}_N^{c_1}\cap \p^{\out}\mathcal{B}_N^{c_1}$, the following estimates hold:
\begin{equation}\label{boundaries}\begin{aligned} &\#( \p^{\out} ( \mathcal{C}_{B^\varepsilon_{r_{i_j}}(\xi_0)}\cap \mathcal{B}_N^{c_1}))\leq \\ \leq& \#( \p^{\out} \mathcal{C}_{B^\varepsilon_{r_{i_j}}(\xi_0)}\cap (\mathcal{B}_N^{c_1})^{\out})+ \#( (\mathcal{C}_{B^\varepsilon_{r_{i_j}}(\xi_0)})^{\out}\cap  \p^{\out}\mathcal{B}_N^{c_1})\\
\leq & \#( \p^{\out} \mathcal{C}_{B^\varepsilon_{r_{i_j}}(\xi_0)}\cap \mathcal{B}_N^{c_1})+ 2 \#( (\mathcal{C}_{B^\varepsilon_{r_{i_j}}(\xi_0)})^{\out}\cap  \p^{\out}\mathcal{B}_N^{c_1}),\end{aligned}\end{equation} and we may consider two possible cases for any index $j\in I_i$.
\begin{itemize}[leftmargin=40pt]
\item[Case 1:] In the case that $$ k_0 \#( \p^{\out} \mathcal{C}_{B^\varepsilon_{r_{i_j}}(\xi_0)}\cap {\mathcal{B}}_N^{c_1})\geq \frac{1}{2}k^{(1-\frac{1}{4D})}e^{\varepsilon D n_i}$$ holds, we first observe that $\p^{\out} \mathcal{C}_{B^\varepsilon_{r_{i_j}}(\xi_0)}\subset \tilde{\mathcal{A}}_{i_j} \cup \mathcal{B}_{\lfloor n_i\rfloor}$ by lemma \ref{separating_sets} and thus by (\ref{cannon_improved})  
\begin{equation}\label{case1}\#(\tilde{\mathcal{A}}_{i_j}\cap \mathcal{B}_N^{c_1}) \geq\frac{ k^{(1-\frac{1}{4D})}}{2  k_0}e^{\varepsilon D n_i}-\tilde C e^{\varepsilon D n_i}.\end{equation}
\item[Case 2:] Otherwise $$ k_0 \#( \p^{\out} \mathcal{C}_{B^\varepsilon_{r_{i_j}}(\xi_0)}\cap \mathcal{B}_N^{c_1})< \frac{1}{2}k^{(1-\frac{1}{4D})}e^{\varepsilon D n_i},$$ so it follows from (\ref{boundaries}) that $$ k^{(1-\frac{1}{4D})}e^{\varepsilon D n_i}\leq  4 k_0\#( (\mathcal{C}_{B^\varepsilon_{r_{i_j}}(\xi_0)})^{\out}\cap  \p^{\out}\mathcal{B}_N^{c_1})$$
and by lemma \ref{ts_estimate} that $$k^{(1-\frac{1}{4D})}e^{\varepsilon D n_i} \leq(24 (\#S)  k_0) \# ({\p^{\out}} {\mathcal{C}}_{B^\varepsilon_{r_{i_j}}(\xi_0)}\cap \mathcal{B}_N^{c_1}).$$ Again, since by lemma \ref{separating_sets} $\p^{\out} {\mathcal{C}}_{B^\varepsilon_{r_{i_j}}(\xi_0)}\subset \tilde{\mathcal{A}}_{i_j} \cup \mathcal{B}_{\lfloor n_i\rfloor}$, it follows by (\ref{cannon_improved}) that \begin{equation}\label{case2}\#(\tilde{\mathcal{A}}_{i_j}\cap \mathcal{B}_N^{c_1})\geq \frac{k^{(1-\frac{1}{4D})}}{24 (\# S) k_0}e^{\varepsilon D n_i} - \tilde C e^{\varepsilon D n_i}.\end{equation}
\end{itemize}
Since $\mathcal{C}_{B^\varepsilon_{r_i}(\xi_0)}\subset \mathcal{C}_{B^\varepsilon_{r_{i_j}}(\xi_0)}$ it follows for all $j$ either from (\ref{case1}) or from (\ref{case2}) by the definition of constant $k$ that $$\#(\tilde{\mathcal{A}}_{i_j}\cap \mathcal{B}_N^{c_1})\geq \tilde C e^{\varepsilon D n_i}$$ and so because $\mathcal{V}_i=\bigsqcup_{j\in I_i}\tilde{\mathcal{A}}_{i_j}$ and by (\ref{eq1}) it holds that $$\#(\mathcal{V}_i\cap \mathcal{B}_N^{c_1})\geq ({\tilde C }e^{\varepsilon D n_i})(k_1 d_i e^{-\varepsilon}e^{\varepsilon n_i}-2)= \left(\frac{6r\tilde C k_1e^{-\varepsilon}} {\pi^2(i+1)^2}e^{\varepsilon \frac{n_i}{2}} -2\tilde C  e^{-\varepsilon \frac{n_i}{2}} \right) e^{\varepsilon (D+\frac{1}{2}) n_{i}}.$$
By the definition of $n_{i+1}$ in (\ref{ni}) we may write $(D+1/4)n_{i+1}=(D+1/2)n_i$ and restate the inequlity above as  \begin{equation}\label{eq3} \#(\mathcal{V}_i\cap \mathcal{B}_N^{c_1})\geq \left(\frac{6r\tilde C  k_1e^{-\varepsilon}} {2\pi^2(i+1)^2}e^{\varepsilon \frac{n_1}{2}\left(\frac{D+1/2}{D+1/4}\right)^{i-1}} -2\tilde C \right) e^{\varepsilon (D+\frac{1}{4}) n_{i+1}}.\end{equation}
Let $\alpha:=(1+4D)^{-1}>0$, so that $\left(\frac{D+1/2}{D+1/4}\right)=1+\alpha$ and estimate $$\frac{1} {(i+1)^2}e^{\varepsilon \frac{n_1}{2}(1+\alpha)^{i-1}} \geq e^{\varepsilon \frac{n_1}{2}}\left(\frac{e^{\varepsilon \frac{n_1}{2}(i-1)\alpha}}{(i+1)^2}\right).$$
By condition (\ref{n1}), we chose $n_1$ large enough that $$\frac{3r \tilde C k_1 e^{-\varepsilon}}{4\pi^2}e^{\varepsilon\frac{n_1}{2}}-2\tilde C \geq k.$$ Observe that the function $(1+x)^{-2}e^{\beta (x-1)}$ is monotone increasing for all $x\geq 0$, whenever $\beta\geq 2$. In particular, since $n_1>32 h\geq 4 (4D+1)\varepsilon^{-1}$, it follows that $\alpha \varepsilon n_1>4$, so for all $j\geq 1$ $$\frac{e^{\varepsilon \frac{n_1}{2}(j-1)\alpha}}{(j+1)^2} \geq \frac{1}{4}.$$
Hence, by (\ref{eq3}) and by condition (\ref{n1})
$$\#(\mathcal{V}_i\cap \mathcal{B}_N^{c_1})\geq \left(3r\pi^{-2}\tilde C k_1 e^{-\varepsilon}e^{\varepsilon \frac{n_1}{2}}\left(\frac{e^{\varepsilon \frac{n_1}{2}(i-1)\alpha}}{(i+1)^2}\right) -2\tilde C \right) e^{\varepsilon (D+\frac{1}{4}) n_{i+1}}\geq k e^{\varepsilon (D+\frac{1}{4}) n_{i+1}}.$$ 
 
Thus we have proved the lemma for $\iota=i+1$. Since $\#(\mathcal{C}_{B^\varepsilon_{r_{i+1}}(\xi_0)}\cap \mathcal{B}_N^{c_1})\geq \#(\mathcal{V}_i\cap \mathcal{B}_N^{c_1})$, we obtain the full statement of the lemma by induction.
\end{proof}

\subsection{The Dirichlet problem at infinity}\label{dp_section}

The main lemma from the previous section has the following corollary. It states that, depending on the radius of the largest ball contained in $D_0$, there is a bound on the distance of $\mathcal{B}_N^{c_0}$ to the identity independent of $N$. Note that the results from this corollary are in a sense a softer version of lemma \ref{values_at_infinity} and as such not essential for the rest of the paper. 

\begin{corollary}\label{thm_transition_sets}
Let $B_{r}^\varepsilon(\xi_0)\subset D_0$ be a ball at infinity and let $x^N$ denote any sequence of minimisers solving the problem defined in proposition \ref{local_minimisers}. Then there exist uniform constants $m\in \N$ and $N_0\in \N$, such that for all $N\geq N_0$ $$\mathcal{B}_N^{c_0}\cap \mathcal{B}_{m}\neq \varnothing.$$
\end{corollary}

\begin{proof}
Let $n_0$ be the smallest natural number satisfying the condition (\ref{n1}) from lemma \ref{main_lemma} and define $m$ as the smallest natural number such that $$m\geq \left(\frac{D+1/2}{D+1/4}\right)\left(1+\frac{1}{4D}\right)\log\left(\frac{\pi^{2D}k}{C_5(6r)^D}\right)n_0.$$
Let $N\in \N$ be such that $\mathcal{B}_N^{c_1}\cap \mathcal{B}_{m}= \mathcal{B}_{m}$, so that by proposition \ref{goc} $$ \# (\mathcal{C}_{B^\varepsilon_{r_1}(\xi_0)}\cap \mathcal{B}_N^{c_1})\geq \# (\mathcal{C}_{B^\varepsilon_{r_1}(\xi_0)}\cap \mathcal{B}_{m})\geq \# (\mathcal{C}_{B^\varepsilon_{r_1}(\xi_0)}\cap S_{m})\geq C_5 r_1^De^{\varepsilon D m}.$$ By the definition of $m$ and because $r_1=6r\pi^{-2}$, it follows for all such $N$ that $$ \# (\mathcal{C}_{B^\varepsilon_{r_1}(\xi_0)}\cap \mathcal{B}_N^{c_1})\geq \left(\frac{D+1/2}{D+1/4}\right) ke^{\varepsilon(D+\frac{1}{4})n_0}.$$ By lemma \ref{main_lemma}, it holds for every $N> m$ and for all $i$ with $n_i<N$ and $n_1$ such that \begin{equation}\label{ts1} n_0\leq n_1\leq \left(\frac{D+1/2}{D+1/4}\right)n_0,\ \text{ that } \ \#(\mathcal{C}_{B^\varepsilon_{r_i}(\xi_0)}\cap \mathcal{B}_N^{c_1})\geq  ke^{\varepsilon(D+\frac{1}{4})n_i} .\end{equation} Moreover, it is easy to see that for any $N>m$ there exists a real number $n_1$ as in \ref{ts1} such that $n_i=N-1$ for some $i$, so it follows that \begin{equation}\label{ts2}\#(\mathcal{C}_{B^\varepsilon_{r_i}(\xi_0)}\cap \mathcal{B}_N^{c_1})\geq ke^{\varepsilon(D+\frac{1}{4})(N-1)}.\end{equation}  On the other hand, by proposition \ref{goc} or by (\ref{cannon_improved}), there exists a uniform constant $\tilde K$ such that $$\#(\mathcal{C}_{B^\varepsilon_{r}(\xi_0)}\cap \mathcal{B}_{N+1})\leq \tilde K e^{\varepsilon D N},$$ which contradicts inequality (\ref{ts2}) for large $N$. More precisely, we may choose $N_0>m$ to be the smallest natural number satisfying $ke^{\frac{\varepsilon}{4} N}> \tilde Ke^{\varepsilon(D+\frac{1}{4})}$. 
\end{proof}

The next step towards the proof of the minimal Dirichlet problem is to show that there exists a neighbourhood $\mathcal{O}\subset (G\cup \p G)$ of the set $\mathring{D}_0\subset \p G$, such that for all $N$ large enough $(\mathcal{O}\cap G) \subset \mathcal{B}_N^{c_0}$. Let us thus define the sets that shall act as such neighbourhoods.

\begin{definition}\label{ui_def}
Let $B_r^{\varepsilon}(\xi_0)$ be a ball at infinity, and 
let $r_i$  and $n_i$ be as in definition \ref{rini}, given $n_1\in \N$. Define for every $i\geq 1$ the set $$\begin{aligned}& \mathcal{U}^i_{B_r^\varepsilon(\xi_0)}:=\{g\in G \ | \ d(g,(\mathcal{C}_{B_{r_i}^\varepsilon(\xi_0)})^c\cup \mathcal{B}_{n_i})\geq \left(\frac{D+1/2}{D+1/4}\right)k e^{\varepsilon (D+\frac{1}{4}) n_{i}}+1\} \\  & \text{and the union } \ \mathcal{U}_{B_r^\varepsilon(\xi_0)}:=\bigcup_{i\geq 1} \mathcal{U}^i_{B_r^\varepsilon(\xi_0)}.\end{aligned}$$
\end{definition}

The following lemma states that asymptotically the boundary of $\mathcal{U}_{B_r^\varepsilon(\xi_0)}$ converges to the boundary of $\mathcal{C}_{B_r^\varepsilon(\xi_0)}$, so it defines a neighbourhood of $B_r^\varepsilon(\xi_0)$ in the visual metric.

\begin{lemma}\label{neighborhood}
Let $B_r^{\varepsilon}(\xi_0)$ be a ball at infinity, let $r_i$  and $n_i$ be as in definition \ref{rini}, with $n_1$ the smallest natural number satisfying (\ref{n1}), and let $\mathcal{U}^i_{B_r^\varepsilon(\xi_0)}$ be as defined above. Then there exists an $i_0\in \N$, such that for all $i\geq i_0$ there is a natural number $M_i$, with $$\mathcal{C}_{B^\varepsilon_{r_i}(\xi_0)}\backslash \mathcal{B}_{M_i}\subset \mathcal{U}^{i+1}_{B_r^\varepsilon(\xi_0)}\subset \mathcal{U}_{B_r^\varepsilon(\xi_0)}.$$ 
\end{lemma}

\begin{proof}
Similarly as at the end of the proof of lemma \ref{main_lemma}, we may choose $i_0$ so large, that for all $i\geq i_0$, $$\frac{e^{\varepsilon n_{i+1}}}{(i+1)^2}\geq \frac{\pi^2C_4}{2r}.$$ For any such fixed $i\geq i_0$ let $m$ be an integer such that $m\geq n_{i+1}.$

It follows by the definition of cones that if $g\in \mathcal{C}_{B^\varepsilon_{r_i}(\xi_0)}\backslash \mathcal{B}_{m}$, then $S(g,R)\subset B^\varepsilon_{r_i}(\xi_0)$. Let $\tilde \xi \in S(g,R)$ and let $\tilde g\in \tilde{\xi}_{\id}$, such that $|\tilde g|=n_{i+1}$. Then, by proposition \ref{shadow}, $S(\tilde g,R)\subset B^\varepsilon_{r_i+C_4e^{-\varepsilon n_{i+1}}}(\xi_0)$. Let now $h\in (\mathcal{C}_{B_{r_{i+1}}^\varepsilon(\xi_0)})^c\cup \mathcal{B}_{n_{i+1}}$. We will show that $d(g,h)\geq m-n_{i+1}$. In case that $h\in \mathcal{B}_{n_{i+1}}$, by the triangle inequality $d(g,h)\geq m- n_{i+1}$. Assume now that $h\in (\mathcal{C}_{B_{r_{i+1}}^\varepsilon(\xi_0)}\cup \mathcal{B}_{n_{i+1}})^c$. Then $|h|\geq n_{i+1}$ and there exists a $\xi \in S(h,R)$ such that $\xi\notin B_{r_{i+1}}^\varepsilon(\xi_0)$. Let $\tilde h\in \xi_{\id}$ be such that $|\tilde h|=n_{i+1}$. It holds that $d^\varepsilon(\xi,\tilde \xi)\geq d_{i}-C_4e^{-\varepsilon n_{i+1}}$ and it follows by assumption on $i_0$ that $$ \frac{1}{2}(d_{i}-C_4e^{-\varepsilon n_{i+1}})   \geq C_4e^{-\varepsilon n_{i+1}}$$ and by proposition \ref{shadow} that $S(\tilde g,R)\cap S(\tilde h,R)=\varnothing$. In particular, $(\mathcal{U}_\xi \cap \mathcal{U}_{\tilde \xi})\backslash \mathcal{B}_{n_{i+1}}=\varnothing$ so $d(\tilde g,\tilde h)\geq 2R$. Since $R\geq 2\tilde \delta$ it easily follows by $\delta$-hyperbolicity that $$d(g,h)\geq m-n_{i+1}.$$
Defining the number $M_i$ as the smallest integer that satisfies \begin{equation}\label{mi}M_i\geq n_{i+1}+k \left(\frac{D+1/2}{D+1/4}\right) e^{\varepsilon (D+\frac{1}{4}) n_{i+1}}+1,\end{equation} it follows that $d(g,h)\geq k \left(\frac{D+1/2}{D+1/4}\right) e^{\varepsilon (D+\frac{1}{4}) n_{i+1}}+1$, which implies that $g\in \mathcal{U}^{i+1}_{B_r^\varepsilon(\xi_0)}$. This finishes the proof of the lemma.

For later reference we note that it now easily follows from definition \ref{ui_def} that for all $m\geq M_i$, \begin{equation}\label{ui_divergence}d(\mathcal{U}^i_{B_r^\varepsilon(\xi_0)}\backslash \mathcal{B}_{m}, (\mathcal{U}^{i+1}_{B_r^\varepsilon(\xi_0)})^c)\geq m-M_i.\end{equation}

\end{proof}

The next statement, which is also a rather direct consequence of the main lemma from the previous section, implies that the values of minimisers $x^N$ on $\mathcal{U}_{B_r^\varepsilon(\xi_0)}$ are uniformly close to the boundary conditions. 

\begin{lemma}\label{values_at_infinity} Let $B_r^{\varepsilon}(\xi_0)\subset D_0$, where $D_0\subset \p G$ is as in proposition \ref{local_minimisers}. Let $r_i$  and $n_i$ be as in definition \ref{rini}, with $n_1$ satisfying (\ref{n1}). Then there exists an integer $N_0$, such that for all $N\geq N_0$ and for all $i$ 
$$\mathcal{U}^i_{B_r^\varepsilon(\xi_0)}\cap \mathcal{B}_N^{c_1}=\varnothing.$$
\end{lemma}

\begin{proof}
Let $n_0$ be the smallest number satisfying condition (\ref{n1}) and let $n_i$ be a sequence as in definition \ref{rini} where $n_1= n_0$. Assume that the lemma is not true, i.e. for some $i\geq 1$ and some $N>n_i+ \left(\frac{D+1/2}{D+1/4}\right)ke^{\varepsilon(D+1/4) n_i}+1$ there exists a $g\in \mathcal{U}^i_{B_r^\varepsilon(\xi_0)}\cap \mathcal{B}_N^{c_1}$. By lemma \ref{connected_components}, there exists a point $\tilde g\in \p^{\out}(\mathcal{B}_N^{c_1}\cap \mathcal{S}_{N+1})$ and a path $p_{g,\tilde g}\subset \mathcal{K}$ from $g$ to $\tilde g$, with $(p_{g,\tilde g}\cap G)\subset \mathcal{B}_N^{c_1}$. Since $\p^{\out}(\mathcal{B}_N^{c_1}\cap \mathcal{S}_{N+1})=((\mathcal{C}_{D_0})^{\inn})^c\cap \mathcal{S}_{N+1}$ and because $B_r^\varepsilon(\xi_0)\subset \mathring D_0$, the path $p_{g,\tilde g}$ intersects $\p^{\inn} (\mathcal{C}_{B_{r_i}^\varepsilon(\xi_0)}\backslash \mathcal{B}_{n_i})$. By the definition of $\mathcal{U}_{B_r^\varepsilon(\xi_0)}$ it then follows that $$\# (\mathcal{B}_N^{c_1}\cap \mathcal{C}_{B^\varepsilon_{r_i}(\xi_0)})\geq \# (p_{g,\tilde g}\cap \mathcal{C}_{B^\varepsilon_{r_i}(\xi_0)})\geq  \left(\frac{D+1/2}{D+1/4}\right) k e^{\varepsilon (D+\frac{1}{4}) n_i}.$$ Similarly as in the proof of corollary \ref{thm_transition_sets}, we may choose a sequence $\tilde n_\iota$ with $n_i\leq \tilde n_i\leq \left(\frac{D+1/2}{D+1/4}\right)n_i$ such that $N=\tilde n_\iota+1$. By lemma \ref{main_lemma} it follows for all $\iota\geq i$ with $\tilde n_\iota<N$ that $\#(\mathcal{C}_{B^\varepsilon_{r_\iota}(\xi_0)}\cap \mathcal{B}_N^{c_1})\geq k e^{\varepsilon (D+\frac{1}{4}) \tilde n_{\iota}} $, which, as in corollary \ref{thm_transition_sets}, drives us to a contradiction about the size of $\mathcal{C}_{B^\varepsilon_{r_i}(\xi_0)}\cap \mathcal{B}_N$ when $N\geq N_0$ and $N_0$ is the smallest natural number with $ke^{\frac{\varepsilon}{4} N}> \tilde Ke^{\varepsilon(D+\frac{1}{4})}$.
\end{proof}

Now we are ready to prove our main theorem. The final step of the proof follows ideas from \cite{anticontinuum}. 

\begin{theorem}\label{dirichlet}
Let $j\in \{0,1\}$ and $\xi_j \in \mathring{D}_j$ be a point at infinity. Let $r>0$ be such that $B^\varepsilon_r(\xi_j)\subset D_j$ and let 
$n_i$ and $r_i$ be as in (\ref{rini}), such that $n_1$ is the smallest number satisfying (\ref{n1}). Let furthermore $i_0$ and $N_0$ be as given in lemmas \ref{neighborhood} and \ref{values_at_infinity}, respectively, and let $M_{i}$ be given by (\ref{mi}). Then for all $n\geq M_{i_0+1}$, all $N\geq N_0$ and all $g\in \mathcal{C}_{B^\varepsilon_{r_{i_0}}(\xi_j)}\backslash \mathcal{B}_{n}$ $$|x_g^N-c_j|\leq \sigma_0\cdot k^{n-M_{i_0+1}}.$$
%

In particular, the global minimiser $\bar x=\lim_{k\to \infty}x^{N_k}$ constructed in proposition \ref{local_minimisers} solves the minimal Cauchy problem at infinity given by $D_0$ and $D_1$ (see definition \ref{dp_def}). 
\end{theorem}

\begin{proof}
We prove the theorem for $j=0$ only, since the other case follows analogously. For any $\xi_0\in \mathring{D}_0$ and $r>0$ as in the statement of the theorem let $\mathcal{U}^i_{B_r^\varepsilon(\xi_0)}$ correspond to $n_i$ and $r_i$, which are again as stated in the theorem. Since $i\geq i_0$, lemma \ref{neighborhood} holds and it follows for all $n\geq M_{i+1}$ that  $$\mathcal{C}_{B^\varepsilon_{r_{i}}(\xi_0)}\backslash \mathcal{B}_{n}\subset \mathcal{U}^{i+1}_{B_r^\varepsilon(\xi_0)}\backslash \mathcal{B}_{n}\subset \mathcal{C}_{B^\varepsilon_{r_{i+1}}(\xi_0)}\backslash \mathcal{B}_{M_{i+1}}\subset \mathcal{U}^{i+2}_{B_r^\varepsilon(\xi_0)},$$ and by (\ref{ui_divergence}) it follows that $$d(\mathcal{C}_{B^\varepsilon_{r_{i}}(\xi_0)}\backslash \mathcal{B}_{n}, (\mathcal{U}^{i+2}_{B_r^\varepsilon(\xi_0)})^c)\geq n-M_{i+1}. $$ Furthermore, $|x_g^N-c_0|\leq \sigma_0$ for any $g\in \mathcal{U}^{i+2}_{B_r^\varepsilon(\xi_0)}$ and any $N\geq N_0$ by lemma \ref{values_at_infinity}.

Recall from theorem \ref{gm_continuations} that for each $N$, $ x^N= x^{\rho}$ for some solution $x:=x^{0}$ of the anti-continuum limit (\ref{ah}), for which $x_g\in \{c_0,c_1\}$ for all $g\in G$. The solution $ x^\rho$ is obtained in theorem \ref{ac_theorem} as a fixed point of a quasi-Newton contraction operator $$K_{\rho,x}(X)_g=X_g - \frac{V'(X_g) + \mathbb{1}_{\mathcal{B}_N}(\rho \Delta_g(X))}{V''(x_g)}$$ with contraction constant $k<1$ on the $\sigma_0$-ball around $x$ in the supremum norm so it follows that $$\|x^\rho - K^m_{\rho,x}(x)\|_{\infty}\leq \sigma_0 k^{m}.$$ Since $\Delta_g(x)$ has range at most one, it follows that $K^m_{\rho,x}(x)_g=c_0$ if $d(\tilde g,\{x_g=c_1\})>m$. By lemma \ref{values_at_infinity} it holds that $x_g=c_0$ for all $g\in \mathcal{U}_{B_r^\varepsilon(\xi_0)}$ and by the inequalities above and the discussion in lemma \ref{values_at_infinity} it follows that $$|x^N_g - c_0|\leq \sigma_0 k^{n-M_{i+1}} \ \text{ for all } \ g\in \mathcal{C}_{B^\varepsilon_{r_{i}}(\xi_0)}\backslash \mathcal{B}_{n}.$$ In particular, for every $\epsilon>0$ there exists a $\tilde n\in \N$, such that $|x_g^N-c_0|\leq \epsilon$ for all $N\in \N$ and for all $g\in \mathcal{C}_{B^\varepsilon_{r_{i_0}}(\xi_0)}\backslash \mathcal{B}_{\tilde n}$, so by definition of $\bar x$ the same holds for such a global minimiser. 

Finally, it is clear by lemma \ref{topology} that convergence with respect to truncated cones $\mathcal{C}_{B^\varepsilon_{r_{i_0}}(\xi_0)}\backslash \mathcal{B}_{\tilde n}$ is equivalent to uniform convergence with respect to the visual metric, which finishes the proof.

\end{proof}

\section{The asymptotic Plateau problem}\label{plateau_section}


In this section we prove a version of the asymptotic Plateau problem for the group $G$. In view of that, let us quantify the size of a boundary of a set $\mathcal{B}\subset G$ by counting all pairs of points $g,gs$ such that $g\in \mathcal{B}$ and $gs \notin \mathcal{B}$, restricted to a finite subset of $G$. More precisely, we define for every set $\mathcal{B}\subset G$ and every finite set $\Omega\subset G$ the function $$b_\Omega(\mathcal{B}):=\#\{(g,gs)\in \Omega\times \Omega \ | \ g\in \mathcal{B}, gs \notin \mathcal{B}\}.$$ One can view $b(\mathcal{B})$ as the number of edges in the the Cayley graph $C(G,S)$ which connect $B$ to its complement. 

\begin{definition}\label{pp_def}
Let $D_0\subset \p G$ and denote $D_1:=\overline{(D_0)^c}$. Assume that $\overline{\mathring{D}}_0=D_0$ (so that also $\overline{\mathring{D}}_1=D_1$ and $\p D_0=\p D_1$). We say that a nonempty set $\mathcal{T}_{D_0}\subset G$ 
solves the asymptotic Plateau problem with respect to $D_0$ if there exist sets $\mathcal{D}_0\subset G$ and $\mathcal{D}_1\subset G$ such that $\mathcal{D}_0\cap\mathcal{D}_1=\varnothing$, $\mathcal{D}_0\cup \mathcal{D}_1=G$ and $\mathcal{T}_{D_0}=\p^{\out} \mathcal{D}_0\cup \p^{\out} \mathcal{D}_1$ has the following properties. 
\begin{itemize}
\item 
Every path $p_{\xi,\eta}\subset C(G,S)$ with $\xi\in \mathring{D}_0$ and $\eta \in \mathring{D}_1$ intersects $\mathcal{T}_{D_0}$.
\item For every two finite set $\Omega \subset G$ and $\tilde{\mathcal{B}}\subset G$ with $\tilde{\mathcal{B}}\cap ({\Omega}^{\inn})^c=\mathcal{D}_0\cap ({\Omega}^{\inn})^c$, $$b_\Omega(\mathcal{D}_0)\leq b_\Omega(\tilde{\mathcal{B}}) \ \text{ and } \ b_\Omega(\mathcal{D}_1)\leq b_\Omega(\tilde{\mathcal{B}}).$$
\end{itemize}
\end{definition}

We use the solutions to the minimal Dirichlet problem for the Allen-Cahn equation to prove the following.

\begin{theorem}\label{plateau}
For every set $D_0\subset \p G$ such that $\overline{\mathring{D}}_0=D_0$ there exists a set $\mathcal{T}_{D_0}\subset G$ solving the asymptotic Plateau problem with respect to ${D_0}$.
\end{theorem}

\begin{proof}
Let $\rho\leq \rho_1$ as in section \ref{anticontinuum} and let $x^\rho$ solve the minimal Dirichlet problem at infinity for the sets $D_0$ and $D_1$, as obtained in theorem \ref{dirichlet}. For every $\rho\leq \rho_1$ define the set $$\mathcal{D}_0^\rho:=\{g\in G \ | \ x^\rho_g\in [c_0, c_0+\sigma_0)\} \ \text{ and } \ \mathcal{D}_1^\rho:=\{g\in G \ | \ x^\rho_g\in (c_1-\sigma_0,c_1]\} $$ and note for later reference that for every finite set $\Omega\subset G$, $b_\Omega(\mathcal{D}^\rho_0)=b_\Omega(\mathcal{D}^\rho_1)$ by theorem \ref{gm_continuations}. According to theorem \ref{ac_theorem}, for every $\rho\leq \rho_1$, a corresponding constant $\sigma\leq \sigma_0$ exists, such that $\|x^\rho-x\|_\infty\leq \sigma$, and such that $\sigma\to 0$ with $\rho\to 0$. Hence, we may fix a decreasing sequence $\rho_1\geq \rho_n\to 0$ such that the corresponding sequence of positive reals $\sigma_n$ is decreasing and such that \begin{equation}\label{condition1}\frac{\hat C e^{hn}}{\hat C e^{hn}+1}<(1-2\sigma_n)^2,\end{equation} where $\hat C:=(\# S) \cdot \tilde C$ and $\tilde C$ is as in (\ref{cannon_improved}). Letting $n\to \infty$ and with that $\rho_n\to 0$, we obtain a sequence of solutions $x^{\rho_n}$ to equations given by (\ref{rr}) and constants $\rho=\rho_n$. By Tychonoff's theorem the set $[c_0,c_1]^G$ is a compact set in the product topology, so there exists a subsequence of $x^{\rho_n}$, converging point-wise to a function $\bar{x}:G\to [c_0,c_1]$. Since $\sigma_n\to 0$, it follows for every $g\in G$, that $\bar{x}_g\in \{c_0,c_1\}$. This allows us to define the sets $$\mathcal{D}_0:=\{g\in G \ | \ \bar{x}_g=c_0\} \ \text{ and } \ \mathcal{D}_1:=\{g\in G \ | \ \bar{x}_g=c_1\}$$ such that $\mathcal{D}_0\cap\mathcal{D}_1=\varnothing$ and $\mathcal{D}_0\cup\mathcal{D}_1=G$. Let $\mathcal{T}_{D_0}:=\p^{\out} \mathcal{D}_0\cup \p^{\out} \mathcal{D}_1$. 

In fact, since for any $m\in \N$ the sequence of functions $x^{\rho_n}$ converges uniformly to $\bar x$ on the ball $\mathcal{B}_m\subset G$, there exists for any $m\in \N$ an integer $\tilde m\in \N$ such that the sets $\mathcal{D}_0^{\rho_n}\cap \mathcal{B}_m$ stabilise for $n\geq \tilde m$. More precisely, $\mathcal{D}_0^{\rho_n}|_{\mathcal{B}_m}=\mathcal{D}_0|_{\mathcal{B}_m}$ for all $n\geq \tilde m$. By taking a subsequence of $\rho_n$, we may thus assume that \begin{equation}\label{condition2}\mathcal{D}_0^{\rho_{\tilde n}}|_{\mathcal{B}_n}=\mathcal{D}_0|_{\mathcal{B}_n} \ \text{for all $n \in \N$ and for all $\tilde n\geq n$.}\end{equation} This obviously implies that also $\mathcal{D}_1^{\rho_{\tilde n}}|_{\mathcal{B}_n}=\mathcal{D}_1|_{\mathcal{B}_n}$ for all $n$ and that $$\mathcal{T}_{D_0}=\lim_{n\to \infty}(\p^{\out} \mathcal{D}_0^{\rho_n}\cup \p^{\out} \mathcal{D}_1^{\rho_n}) \ .$$ 

To confirm that $\mathcal{T}_{D_0}\neq \varnothing$ and that it satisfies the first condition in definition \ref{pp_def}, we first note that all of the constants that appear in section \ref{dp_section} are independent of $\rho$. For any $\xi \in \mathring{D}_0$ and $\eta \in \mathring{D}_1$ we can choose an $r>0$ such that $B^\varepsilon_r(\xi)\subset D_0$ and $B^\varepsilon_r(\eta)\subset D_1$. By theorem \ref{dirichlet}, there exist independently of the constant $\rho$ neighbourhoods $\mathcal{U}_\xi, \mathcal{U}_\eta \subset G$ such that $x^\rho_g\in [c_0,c_0+\sigma)$ for all $g\in \mathcal{U}_\xi$ and $x^\rho_g\in (c_0-\sigma,c_1]$ for all $g\in \mathcal{U}_\eta$. For any such set $\mathcal{U}_\xi, \mathcal{U}_\eta \subset G$ and  for all $n\in \N$, we have $\mathcal{U}_\xi\subset \mathcal{D}_0^{\rho_n}$ and $\mathcal{U}_\eta\subset \mathcal{D}_1^{\rho_n}$. Let $p_{\xi,\eta}\subset C(G,S)$ be any path connecting $\xi$ to $\eta$. Then there exists an $n\in \N$ such that outside of $\mathcal{U}_\eta$ and $\mathcal{U}_\xi$ the path $p_{\xi,\eta}$ is contained in the ball of size $n$, i.e. $(p_{\xi,\eta}\cap G)\backslash (\mathcal{U}_\eta\cup\mathcal{U}_\eta)^{\inn}\subset \mathcal{B}_n$. By condition \ref{condition2} it then holds that $\mathcal{T}_{D_0}|_{\mathcal{B}_n}=(\p^{\out} \mathcal{D}_0^{\rho_n}\cup \p^{\out} \mathcal{D}_1^{\rho_n})|_{\mathcal{B}_n}$ and so clearly $p_{\xi,\eta}\cap \mathcal{T}_{D_0}\cap \mathcal{B}_n\neq \varnothing$.

To investigate the second condition, let $\Omega\subset \mathcal{B}_n\subset G$ be a fixed finite set and observe that for any positive $\rho\leq \rho_1$, \begin{equation}\label{estimate1}W^\rho_\Omega(x)\geq \sum_{g\in \Omega} \frac{\rho^2}{2} \sum_{s\in S}(x_{gs}-x_g)^2\geq \frac{\rho^2}{2} (c_1-c_0)^2(1-2\sigma)^2 b_\Omega(\mathcal{D}_0^\rho) .\end{equation} Let now $\tilde{\mathcal{B}}$ be such that $\tilde{\mathcal{B}}\cap ({\Omega}^{\inn})^c =\mathcal{D}_0^\rho \cap ({\Omega}^{\inn})^c$ and define the function $\tilde{y}$ by $\tilde{y}_g:=c_0$ if $g\in \tilde{\mathcal{B}}$ and $\tilde{y}_g:=c_1$ if $g\notin \tilde{\mathcal{B}}$. For every positive $\rho\leq \rho_1$ it then easily follows $$W_\Omega^\rho(\tilde{y})= \frac{\rho^2}{2} (c_1-c_0)^2 b_\Omega(\tilde{\mathcal{B}}).$$ Since $x^\rho$ is a minimiser with respect to compact variations, $W^\rho_\Omega(x)\leq W^\rho_\Omega(\tilde y)$, and it follows from estimate (\ref{estimate1}) that \begin{equation}\label{ineq1}(1-2\sigma)^2 b_\Omega(\mathcal{D}_0^\rho)\leq b_\Omega(\tilde{\mathcal{B}}).\end{equation}

Since $\Omega\subset \mathcal{B}_n$, $\#\Omega\leq \tilde C e^{hn}$ by (\ref{cannon_improved}) and it follows for every set $\tilde{\mathcal{B}}$ that $b_\Omega(\tilde{\mathcal{B}})\leq \hat C e^{hn}$. Assume now that $\tilde{\mathcal{B}},\hat{\mathcal{B}}\subset G$ are sets such that $\tilde{\mathcal{B}}\cap ({\Omega}^{\inn})^c=\hat{\mathcal{B}}\cap ({\Omega}^{\inn})^c =\mathcal{D}_0^\rho \cap ({\Omega}^{\inn})^c$ and assume that $b_\Omega(\tilde{\mathcal{B}})$ is minimal among all sets with such boundary behaviour with respect to $\Omega$. In case $b_\Omega(\tilde{\mathcal{B}})< b_\Omega(\hat{\mathcal{B}})$ it follows since $ \frac{n}{n+1}$ is a monotone increasing function for $n\in \N$ that \begin{equation}\label{ineq2}\frac{b_\Omega(\tilde{\mathcal{B}})}{b_\Omega(\hat{\mathcal{B}})}\leq \frac{b_\Omega(\tilde{\mathcal{B}})}{b_\Omega(\tilde{\mathcal{B}})+1}\leq \frac{\hat C e^{hn}}{\hat C e^{hn}+1}\ .\end{equation}

Inequality (\ref{ineq1}) and condition (\ref{condition1}) then imply that $$\frac{b_\Omega(\tilde{\mathcal{B}})}{b_\Omega(\mathcal{D}_0^{\rho_n})}\geq (1-2\sigma_n)^2\geq \frac{\hat C e^{hn}}{\hat C e^{hn}+1}$$ and it thus follows by the discussion leading to (\ref{ineq2}) that $b_\Omega(\tilde{\mathcal{B}})\geq b_\Omega(\mathcal{D}_0^{\rho_n})$. Moreover, since $\Omega\subset \mathcal{B}_n$, it follows from condition (\ref{condition2}) that $b_\Omega(\mathcal{D}_0^{\rho_n})=b_\Omega(\mathcal{D}_0)$, so also $b_\Omega(\tilde{\mathcal{B}})\geq b_\Omega(\mathcal{D}_0)$, which implies that $\mathcal{T}_{D_0}$ satisfies also the second condition of definition \ref{pp_def} and finishes the proof.
 
\end{proof}

\begin{remark}
If $\mathcal{T}_{D_0}\subset G$ solves the asymptotic Plateau problem with respect to $D_0$, then every path $p_{\xi,\eta}\subset C(G,S)$ with $\xi\in \mathring{D}_0$ and $\eta \in \mathring{D}_1$ intersects $\mathcal{T}_{D_0}$ and, moreover, this intersection is finite.  
\end{remark}

\begin{remark}
If $\mathcal{T}_{D_0}\subset G$ solves the asymptotic Plateau problem with respect to $D_0$, then all connected components of $\mathcal{D}_0$ and $\mathcal{D}_1$ are infinite. If not, there would be a finite connected component contained in a finite ball $\mathcal{B}_n$ given by the transition set of the minimal solution $x^{\rho_n}$ as in the proof of theorem \ref{plateau}, which contradicts lemma \ref{connected_components}. 
\end{remark}

\newpage

%
\appendix
\section{Proof of theorem \ref{ac_theorem}}
\begin{proof} Observe that a function $X:G\to \R$ is a solution to (\ref{rr}) on $\mathcal{B}^{\inn}$ with given boundary data if and only if $F(X, \rho)=0$, where the function $F:l_\infty(G)\times \R \to l_\infty(G)$ is defined by 
$$ F(X, \rho)_g:=V'(X_g)+\mathbb{1}_{\mathcal{B}^{\inn}} (\rho\Delta_g(X)).$$
In particular, $F(x,0)=0$ for any $x$ that solves (\ref{ah}). One now wants to apply the implicit function theorem to conclude the existence of a family  $x^\rho$ near $x$ with $F(x^\rho, \rho)=0$. 


Let $v\in l_\infty(G)$ and observe that the Fr\'echet derivative $D_XF(x,0): l_{\infty} \to l_{\infty}$ at $x$ is given by
\[
\left( D_XF(x,0)\cdot v \right)_g:=\left.\frac{d}{dt}\right|_{t=0}\!\! F(x+tv,0)_g = V''(x_g) \cdot v_g\, .
\] 
Because $V$ is a Morse function and $x_g\in [c_0,c_1]$, $|V''(x_g)| > \hat c$ for some constant $\hat c>0$ and for all $g\in G$. It follows that $D_XF(x,0): l_{\infty}\to l_{\infty}$ has a bounded inverse and we may define the quasi-Newton operator $K_{\rho, x}$ by 
$$K_{\rho, x}(X):=X-D_XF(x,0)^{-1}\cdot F(X,\rho), \ \mbox{i.e.}\ K_{\rho, x}(X)_g:= X_g - \frac{V'(X_g) + \mathbb{1}_{\mathcal{B}^{\inn}} (\rho\Delta_g(X))}{V''(x_g)}\ .$$
Obviously, $K_{\rho, x}$ maps $l_\infty$ into itself and it is clear that $K_{\rho, x}(X)=X$ if and only if $F(X,\rho)=0$. Restricted to an appropriately chosen small $l_\infty$ ball $\{X \ | \ \|X-x\|_{\infty}\leq \sigma_0\}$ around $x$, the operator $K_{\rho, x}$ moreover acts as a very strong contraction that sends this ball into itself. 

This can for example be seen from the following standard argument. Let us choose any desired contraction constant $0<k<1$ and, accordingly, a $0<\sigma_0<\frac{1}{2}$ with the property that $|V''(X_g)-V''(x_g)|\leq \frac{k\hat c}{2}$ uniformly for $X\in \{X \ | \ \|X-x\|_{\infty}\leq \sigma_0\}$. Such $\sigma_0$ exists because $V''$ is assumed continuous and because $V$ only has finitely many stationary points. For example, when $V''$ is Lipschitz continuous with Lipschitz constant $L$, then it suffices to choose $\sigma_0=\frac{k\hat c}{2L}$. For later reference, let us remark that it holds automatically that the intervals
\begin{align}
\label{disjointintervals}
 (x_g-\sigma_0, x_g+\sigma_0) \cap (x_{\tilde g}-\sigma_0, x_{\tilde g}+\sigma_0) =\varnothing\ \mbox{if} \ x_g\neq x_{\tilde g}\ \mbox{are critical points of}\ V\, .
\end{align}

We now have for $g\in \mathcal{B}^{\inn}$ and $X,Y$ with $\|X-x\|_{\infty}\leq \sigma_0$ and $\|Y-x\|_{\infty}\leq \sigma_0$, that
\begin{align}\label{contraction}
|K_{\rho, x}(X)_g-K_{\rho, x}(Y)_g| \leq \left|X_g-Y_g -\frac{V'(X_g)-V'(Y_g)}{V''(x_g)}\right| + \left|\frac{\rho \Delta_g(X) -\rho \Delta_g(Y)}{V''(x_g)}\right|\, .
\end{align}
To estimate the first term in (\ref{contraction}) we write
$$X_g-Y_g - \frac{V'(X_g)-V'(Y_g)}{V''(x_g)}= \frac{X_g-Y_g}{V''(x_g)}\left(\int_0^1\!V''(x_g)-V''(tX_g+(1-t)Y_g)\ \!dt\right)  \, ,$$  
which shows that the first term in (\ref{contraction}) is bounded from above by $\frac{k}{2}||X-Y||_{\infty}$. To estimate the second part of the sum in (\ref{contraction}), we note that since $\|X-x\|_{\infty}\leq \sigma_0$ and because $x_g\in [c_0,c_1]$ for all $g$, it holds that $|X_{gs}-X_g|\leq c_1-c_0+2\sigma_0$, and similarly for $Y$. It then follows by linearity of $\Delta$ that $$|\Delta_g(X) -\Delta_g(Y)|\leq (c_1-c_0+2\sigma_0)2(\#S)\|X-Y\|_\infty.$$
By defining $\tilde c:=(c_1-c_0+2\sigma_0)2(\#S)$, we found that it holds for $X,Y$ as above that
 $$||K_{\rho, x}(X)-K_{\rho, x}(Y)||_{\infty}\leq \left(\frac{k}{2}+\frac{\rho\cdot\tilde c}{\hat c}\right) ||X-Y||_{\infty}\ .$$
To investigate whether $K_{\rho, x}$ maps the $\sigma_0$-ball around $x$ to itself, let $||X-x||_{\infty}< \sigma_0$ and observe that it follows that
$$\begin{aligned}|K_{\rho, x}(X)_g-x_g|\leq |K_{\rho, x}(X)_g-K_{\rho, x}(x)_g|+ |K_{\rho, x}(x)_g - x_g |<  \sigma_0\left(\frac{k}{2}+\frac{\rho\cdot\tilde c}{\hat c}\right)+ \left(\frac{ \rho\cdot\tilde c}{\hat c}\right)\ ,\end{aligned}$$
where the final estimate holds because $V'(x_g)=0$ and $|\Delta_g(x)| \leq \#S(c_1-c_0)<\tilde c$.

This proves that $K_{\rho, x}$ is a contraction with constant $k$ which maps the $\sigma_0$-ball around $x$ to itself if $0\leq \rho\leq \rho_0 := \min\left\{\frac{k\hat c}{2\tilde c}, \frac{(1-k)\sigma_0\hat c}{\tilde c}\right\}$. Clearly then $\lim_{\sigma_0\searrow 0}\rho_0=0$.
\end{proof}

\begin{small}
\bibliographystyle{amsplain}
\bibliography{anti-harmonic}
\end{small}

\end{document}